\documentclass{article}
\usepackage{fixltx2e}
\usepackage{amssymb}
\usepackage{amsfonts}
\usepackage{amsmath}
\usepackage[ansinew]{inputenc}
\usepackage{srcltx}
\usepackage[T1]{fontenc}
\usepackage{lmodern}
\usepackage{exscale}
\usepackage{graphicx}

\newtheorem{theorem}{Theorem}

\newtheorem{definition}[theorem]{Definition}
\newtheorem{example}[theorem]{Example}

\newtheorem{lemma}[theorem]{Lemma}

\newtheorem{remark}[theorem]{Remark}

\newenvironment{proof}[1][Proof]{\noindent\textbf{#1.} }{\ \rule{0.5em}{0.5em}}
\renewcommand\sb{\subseteq}
\newcommand\sm{\setminus}
\newcommand\mto{\mapsto}
\newcommand\lmto{\longmapsto}

\newcommand\cC{\mathcal{C}}

\newcommand\cS{\mathcal{S}}
\renewcommand\atop[2]{\genfrac{}{}{0pt}{}{#1}{#2}}
\newcommand{\bbinom}[2]{\genfrac{[}{]}{0pt}{}{#1}{#2}}
\newcommand{\ebinom}[2]{\genfrac{\langle}{\rangle}{0pt}{}{#1}{#2}}
\DeclareMathOperator\sn{sn}
\DeclareMathOperator\des{des}
\DeclareMathOperator\cdes{cdes}
\DeclareMathOperator\cs{\mathfrak{s}}
\DeclareMathOperator\cl{\mathfrak{l}}
\DeclareMathOperator\crr{\mathfrak{r}}

\DeclareMathOperator\sech{sech}
\begin{document}

\title{Moments of quantum L\'{e}vy areas using sticky shuffle
Hopf algebras}
\author{Robin Hudson \\
Loughborough University,\\
Loughborough LE11 3TU\\
Great Britain \and Uwe Schauz and Wu Yue \\
Xi'an Jiaotong-Liverpool University,\\
Suzhou,\\
China}
\maketitle

\begin{abstract}
We study a family of quantum analogs of L\'{e}vy's stochastic area for planar
Brownian motion depending on a variance parameter $\sigma \geq 1$ which deform
to the classical L\'{e}vy area as $\sigma \rightarrow \infty .$ They are defined as
second rank iterated stochastic integrals against the components of planar Brownian
motion, which are one-dimensional Brownian motions satisfying Heisnberg-type
commutation relations. Such iterated integrals\ can be multiplied using the sticky
shuffle product determined by the underlying It\^{o} algebra of stochastic differentials.
We use the corresponding Hopf algebra structure to evaluate the moments of the
quantum L\'{e}vy areas and study how they deform \ to their classical values, which
are well known to be given essentially by the Euler numbers, in the infinite variance
limit.


\end{abstract}

\section{Introduction.}

L\'{e}vy's stochastic area for planar Brownian motion is important in several areas of
modern mathematics and probability theory, ranging from harmonic analysis on the
Heisenberg group to rough noise analysis.

Let us first review the definition of L\'{e}vy area as a stochastic integral \cite{Levy1}.
Intuitively it is the signed area between the chord joining two time points on the planar
Brownian path and the trajectory between those
points. To make this rigorous, let there be given a planar Brownian motion $%
B $ and write $B=\left( X,Y\right) $ in terms of components $X$ and $Y$ which are
independent one-dimensional Brownian motions. Let two real numbers $a<b$ be
given.
\begin{definition}\label{classicLA}
 \emph{The L\'{e}vy area of }$B$\emph{\ over the time
interval }$[a,b)$\emph{\ is the stochastic integral }%
\begin{equation*}
\mathcal{A}_{[a,b)}=\dfrac{1}{2}\int_{a}^{b}\Bigl(\left( X-X(a)
\right) dY-\left( Y-Y( a) \right) dX\Bigr).  \label{Levyarea}
\end{equation*}
\end{definition}

In this definition the integral takes the same value whether it is regarded as of It\^{o} or
Stratonovich type, but in the remainder of this paper all stochastic integrals will be of
It\^{o} type, in contrast to \cite{LeWi} where the Stratonovich integral is used. The latter
cannot be defined coherently in a quantum context.

L\'{e}vy area has interesting connections with classical mathematics through its
characteristic function which is given by the following theorem.

\begin{theorem}\label{Levy Sech}
(L\'{e}vy \cite{Levy})
\begin{equation*}
\mathbb{E}\!\left[\exp \left(iz\mathcal{A}_{[a,b)}\right)\right]\,=\,\sech\!\left(%
\tfrac{1}{2}%
\left( b-a\right) z\right).  
\end{equation*}
\end{theorem}

We can expand the right-hand side of the formula in Theorem\,\ref{Levy Sech} using
the Taylor series
\begin{equation}
\sech(z)\,=\,\sum_{m=0}^{\infty }\left( -1\right) ^{m}\frac{A_{2m}}{\left(
2m\right) !}z^{2m}\ ,
\end{equation}%
where the even Euler zigzag numbers $A_{2m}$ are related to the Riemann zeta 
function $\zeta $ by%
\begin{equation}
\zeta \left( 2m\right) =\frac{\pi ^{2m}}{\left( 2m\right) !}A_{2m}.
\end{equation}

Levin and Wildon in \cite{LeWi} used iterated integrals and combinatorial
arguments arising from the formalism of rough noise to evaluate the moments of $\mathcal{%
A}_{[0,1)},$ which is tantamount to proving Theorem\,\ref{Levy Sech}.

A one-parameter family of \emph{quantum L\'{e}vy areas} has been introduced recently
\cite{Huds2,ChHu}. In these the component one-dimensional Brownian motions of the
classical L\'{e}vy area are replaced by a pair of self adjoint operator-valued processes
$\left( P^{\left( \sigma \right) }\left( t\right) ,Q^{\left( \sigma \right) }\left( t\right) \right)
_{t\geq 0}.$ Each such pair is determined by a variance parameter $\sigma $ taking a
value in the range $1\leq \sigma \,<\infty .$ Each of $P^{\left( \sigma \right) }$ and
$Q^{\left( \sigma \right) }$ is individually a one-dimensional Brownian motion of
variance $\sigma ^{2},$ so that for example, for each positive time $t,$ $P^{\left( \sigma
\right) }\left( t\right) $ is a Gaussian random variable of mean $0$ and variance $\sigma
^{2}t$. But the processes $P^{\left( \sigma \right) }$ and $Q^{\left( \sigma \right) }$ do
not commute with each other; instead they satisfy the
Heisenberg type commutation relation%
\begin{equation}
\left[ P^{\left( \sigma \right) }\left( s\right) ,Q^{\left( \sigma \right)
}\left( t\right) \right] =-2i\min \left\{ s,t\right\}   \label{Heisenberg}
\end{equation}%
in the rigorous Weyl sense that for arbitrary real $x$ and $y$ and
nonnegative $s$ and $t,$%
\begin{equation}
e^{ixP^{\left( \sigma \right) }\left( s\right) }e^{iyQ^{\left( \sigma
\right) }\left( t\right) }=e^{2ixy\min \left\{ s,t\right\} }e^{iyQ^{\left(
\sigma \right) }\left( t\right) }e^{ixP^{\left( \sigma \right) }\left(
s\right) }
\end{equation}%
as unitary operators. Despite their mutual noncommutativity $P^{\left( \sigma \right) }$
and $Q^{\left( \sigma \right) }$ can be regarded as stochastically independent in a
certain sense, and hence as the two components of a \emph{quantum planar Brownian
motion. }Indeed\emph{,} for arbitrary real $x_{1},x_{2},...,x_{m},$ $y_{1},y_{2},...,y_{n}$
and
nonnegative $s_{1},s_{2},...,s_{m},$ $t_{1},t_{2},...,t_{n},$ the operator $%
\sum x_{j}P^{\left( \sigma \right) }\left( s_{j}\right) $ $+\sum y_{k}Q^{\left( \sigma \right)
}\left( t_{k}\right) $ defined on the intersection of domains is essentially self-adjoint, so
that the quantum probabilistic expectation $\mathbb{E}\left[ e^{i\left( \sum x_{j}P^{\left(
\sigma \right) }\left( s_{j}\right) +\sum y_{k}Q^{\left( \sigma \right) }\left( t_{k}\right) \right)
}\right] $, which in effect determines the \emph{joint characteristic function}, is well
defined. Moreover this factorizes:
\begin{equation}
\mathbb{E}\left[ e^{i\left( \sum x_{j}P^{\left( \sigma \right) }\left(
s_{j}\right) +\sum y_{k}Q^{\left( \sigma \right) }\left( t_{k}\right)
\right) }\right]\, \,=\mathbb{E}\left[ e^{i\sum x_{j}P^{\left( \sigma \right)
}\left( s_{j}\right) }\right] \mathbb{E}\left[ e^{i\sum y_{k}Q^{\left(
\sigma \right) }\left( t_{k}\right) }\right]
\end{equation}%
and in classical probability such factorization is sufficient for independence.

We use the standard quantum stochastic calculus of \cite{Part} in the case when
$\sigma =1$ and the non-Fock finite temperature calculus of \cite{HuLi} when $\sigma
>1$ to define the corresponding quantum L\'{e}vy areas, in
which the planar Brownian motion is replaced by its quantum version $%
R^{\left( \sigma \right) }=\left( P^{\left( \sigma \right) },Q^{\left( \sigma \right) }\right) .$
\begin{definition}\label{QuantumLA}
\emph{The quantum} \emph{L\'{e}vy area }$\mathcal{B}%
_{[a,b)}^{\left( \sigma \right) }$\emph{\ of }$R^{\left( \sigma \right) }$%
\emph{\ of variance }$\sigma $ \emph{over the time interval }$[a,b)$ \emph{is the
quantum stochastic integral }%
\begin{equation*}
\mathcal{B}_{[a,b)}^{\left( \sigma \right) }\,=\,\dfrac{1}{2}\int_{a}^{b}\left((
P^{\left( \sigma \right) }-P^{\left( \sigma \right) }(a))dQ^{\left( \sigma
\right) }-( Q^{\left( \sigma \right) }-Q^{\left( \sigma \right)
}(a)) dP^{\left( \sigma \right) }\right) .  
\end{equation*}%
\end{definition}

When $\sigma =1,$
the distribution at all times of the corresponding L\'{e}vy areas is degenerate at $t=0$
and all moments are zero \cite{ChHu}. For values $\sigma >1$ the processes $R^{\left(
\sigma \right) }
$ generate Type III factorial von Neumann algebras\footnote{%
More precisely, it is the the unitary operators $e^{ixP^{\left( \sigma \right) }\left( s\right)
}$ and $e^{ixQ^{\left( \sigma \right) }\left( s\right) }$ for real $x$ and positive $\sigma $
which generate these algebras.}, whose mutual strong unitary inequivalence as
$\sigma $ varies can be regarded as a quantum version of the mutual singularity of the
measures obtained by dilatation of planar Wiener measure through different dilatation
factors $\sigma .$

In view of \ (\ref{Heisenberg}), the normalised standard unit variance Brownian motions
\begin{equation}
\hat{P}^{\left( \sigma \right) }=\sigma ^{-1}P^{\left( \sigma \right) }\text{%
, }\hat{Q}^{\left( \sigma \right) }=\sigma ^{-1}Q^{\left( \sigma \right) },
\end{equation}%
become mutually commutative in the limit of large $\sigma ,$ so that the corresponding
quantum L\'{e}vy areas
\begin{equation}
\mathcal{\hat{B}}_{[a,b)}^{\left( \sigma \right) }=\sigma ^{-2}\mathcal{B}%
_{[a,b)}^{\left( \sigma \right) }
\end{equation}%
interpolate between the degenerate distribution at $\sigma =1$ and the classical case
$\mathcal{A}_{[a,b)}$ at $\infty .$ Thus it is a natural question to ask how the moments
behave under this interpolation and in particular how the Euler-zigzag numbers are
approached at $\infty $. The main purpose of this paper is to address this question.

Our method is based firstly on the observation that Definition\,\ref{classicLA},
Definition\,\ref{QuantumLA}, and the normalized form of the latter,
can be regarded as \emph{iterated} stochastic integrals:%
\begin{eqnarray}
\mathcal{A}_{[a,b)} &=&\dfrac{1}{2}\int_{a\,<x<\,y<\,b}\Big(
dX(x)dY(y)-dY\left( x\right) dX\left( y\right) \Big) ,  \label{classical}
\\
\mathcal{B}_{[a,b)}^{\left( \sigma \right) } &=&\dfrac{1}{2}%
\int_{a\,<x<\,y<\,b}\left( dP^{\left( \sigma \right) }\left( x\right)
dQ^{\left( \sigma \right) }\left( y\right) -dQ^{\left( \sigma \right)
}\left( x\right) dP^{\left( \sigma \right) }\left( y\right) \right) ,
\label{quantum} \\
\mathcal{\hat{B}}_{[a,b)}^{\left( \sigma \right) } &=&\dfrac{1}{2}%
\int_{a\,<x<\,y<\,b}\left( d\hat{P}^{\left( \sigma \right) }\left( x\right) d%
\hat{Q}^{\left( \sigma \right) }\left( y\right) -d\hat{Q}^{\left( \sigma
\right) }\left( x\right) d\hat{P}^{\left( \sigma \right) }\left( y\right)
\right) .  \label{quantumhat}
\end{eqnarray}%
We may thus evaluate moments as expectations of powers, using the so-called
\emph{sticky shuffle }\cite{Huds1} or \emph{stuffle }\cite{Hoff} Hopf algebra.
Multiplication in this algebra can be used to express the product of two iterated It\^{o}
stochastic integrals as a linear combination of such iterated integrals. Since the
expectation of an iterated integral vanishes unless each of the individual integrators is
time, the \emph{recovery formula} \cite{Huds1,Bour}\ involving higher order Hopf 
algebra coproducts reduces the evaluation of the moments to a combinatorial counting
problem.

The sticky shuffle Hopf algebra is reviewed in Section 2 and its use for reducing the
evaluation of\ moments to a counting problem is described in Section 3. We review
some combinatorial results needed to accomplish this counting in
Section 4. In section 5 we evaluate the moments of the quantum L\'{e}vy area (\ref%
{quantumhat}). Finally in Section 6 we show how the classical moments \cite{LeWi} are
recovered in the "infinite temperature" limit as $\sigma \rightarrow \infty .$

\section{The sticky shuffle product Hopf algebra} 
\subsection{The shuffle product Hopf algebra}
Given a complex vector space $\mathcal{L}$,\ the usual \textit{shuffle product Hopf
algebra} over $\mathcal{L}$\ is formed by equipping the vector space
$\mathcal{T}(\mathcal{L})=\bigoplus\limits_{n=0}^{\infty
} 
\bigotimes\limits_{j=1}^{n}\mathcal{L}\ 
$ 
of tensors of all ranks over $\mathcal{L}$ with the operations of product, unit,
coproduct, counit and antipode defined as follows. We denote a general element
$\alpha$ of $\mathcal{T}\left( \mathcal{L}\right)$ by $\alpha =\alpha _{0}\oplus \alpha
_{1}\oplus \alpha _{2}\oplus \cdots $ or $\left( \alpha _{0},\alpha _{1},\alpha
_{2},...\right),$ where only finitely many of the $\alpha _{m}$
are nonzero. For each $\alpha _{m}\in\bigotimes \limits_{j=1}^{m}%
\mathcal{L}$ the corresponding embedded element $\left( 0,0,...,\alpha _{m},0,...\right)
$ of $\mathcal{T}\left( \mathcal{L}\right) $ is denoted by $\left\{ \alpha _{m}\right\} .$

\begin{itemize}
\item The \emph{shuffle product} is defined by bilinear extension of the rule%
\begin{eqnarray}
&&\!\!\!\!\!\!\!\!\left\{ L_{1}\otimes L_{2}\otimes \cdots \otimes L_{m}\right\} \left\{
L_{m+1}\otimes L_{m+2}\otimes \cdots \otimes L_{m+n}\right\}   \notag \\
&=&\sum\limits_{\mathfrak{s}\in \mathcal{S}\left( m,n\right) }\left\{ L_{%
\mathfrak{s}\left( 1\right) }\otimes L_{\mathfrak{s}\left( 2\right) }\otimes
\cdots \otimes L_{\mathfrak{s}(m+n)}\right\}   \label{shuffle}
\end{eqnarray}%
where $\mathcal{S}\left( m,n\right) $ denotes the set of $(m,n)$-\emph{%
shuffles}, that is permutations $\mathfrak{s}$ of $\left\{
1,2,...,m+n\right\} $ for which $\mathfrak{s}\left( 1\right) <\mathfrak{s}%
\left( 2\right) <\cdots <\mathfrak{s}\left( m\right) $ and $\mathfrak{s}%
\left( m+1\right) <\mathfrak{s}\left( m+2\right) <\cdots <\mathfrak{s}\left( m+n\right) .$

\item The \textit{unit element} for this product is $1_{\mathcal{T}\left( \mathcal{L}\right)
    }=\left( 1_{\mathbb{C}},0,0,...\right) $.

\item The \textit{coproduct} $\Delta $ is the map from $\mathcal{T}\left(
    \mathcal{L}\right) $ to $\mathcal{T}\left( \mathcal{L}\right) \otimes \mathcal{T}\left(
    \mathcal{L}\right) $ defined by linear extension of the rules that $\Delta \left(
    1_{\mathcal{T}\left( \mathcal{L}\right) }\right) =1_{\mathcal{T}\left( \mathcal{L}\right)
    }\otimes 1_{\mathcal{T}\left( \mathcal{L}\right) }=1_{\mathcal{T}\left(
    \mathcal{L}\right) \otimes \mathcal{T}\left( \mathcal{L}\right) }$ and
\begin{eqnarray}
&&\!\!\!\!\!\!\!\!\Delta \left\{ L_{1}\otimes L_{2}\otimes \cdots \otimes L_{m}\right\} \nonumber\\
&=&1_{\mathcal{T}\left( \mathcal{L}\right) }\otimes \left\{ L_{1}\otimes
L_{2}\otimes \cdots \otimes L_{m}\right\}\nonumber \\
&&+\sum\limits_{j=2}^{m}\left\{ L_{1}\otimes L_{2}\otimes \cdots \otimes
L_{j-1}\right\} \otimes \left\{ L_{j}\otimes L_{j+1}\otimes \cdots \otimes
L_{m}\right\} \nonumber\\
&&+\left\{ L_{1}\otimes L_{2}\otimes \cdots \otimes L_{m}\right\} \otimes 1_{%
\mathcal{T}\left( \mathcal{L}\right) } .
\end{eqnarray}

\item The \textit{counit} $\varepsilon $ is the map from $\mathcal{T}\left(
\mathcal{L}\right) $ to $\mathbb{C}$\ defined by linear extension of%
\begin{equation}
\varepsilon \left( 1_{\mathcal{T}\left( \mathcal{L}\right) }\right) =1_{%
\mathbb{C}}\,,\text{ }\varepsilon \left\{ L_{1}\otimes L_{2}\otimes \cdots
\otimes L_{m}\right\} =0\,\text{ for }m>0.
\end{equation}

\item The \textit{antipode} is the map $S$ from\textit{\ }$\mathcal{T}\left(
    \mathcal{L}\right) $ to $\mathcal{T}\left( \mathcal{L}\right) $ defined by
linear extension of%
\begin{eqnarray}
S\left( 1_{\mathcal{T}\left( \mathcal{L}\right) }\right) &=&1_{\mathcal{T}%
\left( \mathcal{L}\right) }, \\
\text{ }S\left\{ L_{1}\otimes L_{2}\otimes \cdots \otimes L_{m}\right\}
&=&(-1)^{m}\left\{ L_{m}\otimes L_{m-1}\otimes \cdots \otimes L_{1}\right\}\quad
\end{eqnarray}%
for $m>0.$
\end{itemize}

There are two useful equivalent definitions of the shuffle product. For the first, we use
the notational convention that, for arbitrary
elements $\alpha $ of $\mathcal{T}\left( \mathcal{L}\right) $ and $L$ of $%
\mathcal{L,}$ $\alpha \otimes L$ is the element of $\mathcal{T}\left(
\mathcal{L}\right) $ for which $\left( \alpha \otimes L\right) _{0}=0$ and $%
\left( \alpha \otimes L\right) _{n}=\alpha _{n-1}\otimes L$ for $n\geq 1.$ Then the shuffle
product of arbitrary elements of $\mathcal{T}\left( \mathcal{L}\right) $ is defined
inductively by bilinear extension of the
rules%
\begin{eqnarray}
&&\!\!\!\!\!\!\!\!1_{\mathcal{T}\left( \mathcal{L}\right) }\left\{ L_{1}\otimes L_{2}\otimes
\cdots \otimes L_{m}\right\} =\left\{ L_{1}\otimes L_{2}\otimes \cdots
\otimes L_{m}\right\} 1_{\mathcal{T}\left( \mathcal{L}\right) } \notag\\
&=&\left\{ L_{1}\otimes L_{2}\otimes \cdots \otimes L_{m}\right\} ,
\end{eqnarray}%
\begin{eqnarray}\label{SSP}
&&\!\!\!\!\!\!\!\!\left\{ L_{1}\otimes L_{2}\otimes \cdots \otimes L_{m}\right\} \left\{
L_{m+1}\otimes L_{m+2}\otimes \cdots \otimes L_{m+n}\right\} \notag \\
&=&\left( \left\{ L_{1}\otimes L_{2}\otimes \cdots \otimes L_{m-1}\right\}
\left\{ L_{m+1}\otimes L_{m+2}\otimes \cdots \otimes L_{m+n}\right\} \right)
\otimes L_{m}  \\
&&+\left( \left\{ L_{1}\otimes L_{2}\otimes \cdots \otimes L_{m}\right\}
\left\{ L_{m+1}\otimes L_{m+2}\otimes \cdots \otimes L_{m+n-1}\right\}
\right) \otimes L_{m+n}.  \notag
\end{eqnarray}%
Here the two terms on the right-hand side of (\ref{SSP}) correspond to the mutually
exclusive and exhaustive possibilities that $\mathfrak{s}(m+n)=m$ and
$\mathfrak{s}(m+n)=m+n$ in the expansion (\ref{shuffle}). The second alternative
definition is that the shuffle product $\gamma =\alpha \beta $ of arbitrary elements
$\alpha $ and $\beta $ is given by
\begin{equation}\label{alternative}
\gamma _{N}\,=\sum_{\stackrel{A\cup B=\{1,2,...,N\}}{A\cap B=\emptyset} }\alpha
_{\left\vert A\right\vert }^{A}\beta _{\left\vert B\right\vert }^{B}.
\end{equation}%
Here the sum is over the $2^{N}$ ordered pairs $(A,B)$ of disjoint subsets
whose union is $\left\{ 1,2,...,N\right\} $ and the notation is as follows; $%
\left\vert A\right\vert $ denotes the number of elements in the set $A$ so that $\alpha
_{\left\vert A\right\vert }$ denotes the homogeneous component of rank $\left\vert
A\right\vert $ of the tensor $\alpha =\left( \alpha _{0},\alpha _{1},\alpha _{2},...\right) ,$
and $\alpha _{\left\vert A\right\vert }^{A}$ indicates that this component is to be
regarded as occupying only those $\left\vert A\right\vert $ copies of $\mathcal{L}$
within $\bigotimes\limits_{j=1}^{N}\mathcal{L}$ labelled by elements of the subset $A$
of $\{1,2,...,N\}.$ Thus with $\beta _{\left\vert B\right\vert }^{B}$ defined analogously the
combination $\alpha _{\left\vert A\right\vert
}^{A}\beta _{\left\vert B\right\vert }^{B}$ is a well-defined element of $%
\bigotimes\limits_{j=1}^{N}\mathcal{L}$.
\subsection{The sticky shuffle algebra}
Now suppose that the complex vector space $\mathcal{L}$ is an associative
algebra. We define the \textit{sticky shuffle product} in the vector space $%
\mathcal{T}\left( \mathcal{L}\right) $ by modifying definition (\ref%
{SSP}) by inserting an extra term so that now%
\begin{eqnarray}\label{sticky}
&&\!\!\!\!\!\!\!\!\!\!\!\!\!\left\{ L_{1}\otimes L_{2}\otimes \cdots \otimes L_{m}\right\} \left\{
L_{m+1}\otimes L_{m+2}\otimes \cdots \otimes L_{m+n}\right\}  \notag \\
&=&\left( \left\{ L_{1}\otimes \cdots \otimes L_{m-1}\right\}
\left\{ L_{m+1}\otimes \cdots \otimes L_{m+n}\right\} \right)
\otimes L_{m}  \notag \\
&&+\left( \left\{ L_{1}\otimes \cdots \otimes L_{m}\right\}
\left\{ L_{m+1}\otimes \cdots \otimes L_{m+n-1}\right\}
\right) \otimes L_{m+n}  \label{sticky} \\
&&+\left( \left\{ L_{1}\otimes \cdots \otimes L_{m-1}\right\}
\left\{ L_{m+1}\otimes \cdots \otimes L_{m+n-1}\right\}
\right) \otimes L_{m}L_{m+n}.  \notag
\end{eqnarray}%
Or we can modify the alternative definition of the shuffle product (\ref%
{alternative}) by defining the product $\gamma =\alpha \beta $ by%
\begin{equation}\label{shuffle'}
\gamma _{N}\,=\sum_{A\cup B=\{1,2,...,N\}}\alpha _{\left\vert A\right\vert
}^{A}\beta _{\left\vert B\right\vert }^{B}.
\end{equation}%
Here the sum is now over the $3^{N}$ not necessarily disjoint ordered pairs $%
(A,B)$ whose union is $\{1,2,...,N\},$ $\alpha _{\left\vert A\right\vert }^{A}$ and $\beta
_{\left\vert B\right\vert }^{B}$ are defined as before but now if $A\cap B\neq \emptyset $
double occupancy of a copy of $\mathcal{L}$ within
$\bigotimes\limits_{j=1}^{n}\mathcal{L}$ is reduced to single occupancy by using the
multiplication in the algebra $\mathcal{L}$ as a map from $\mathcal{L\times L}$ to
$\mathcal{L.}$ Thus the sticky shuffle product
reduces to the usual shuffle product in the case when the multiplication in $%
\mathcal{L}$ is trivial with all products vanishing. That (\ref{shuffle'}) is equivalent to
(\ref{sticky}) (and in particular, that (\ref{alternative}) is equivalent to (\ref{SSP})) is seen
by noting that the three terms on the right hand side of (\ref{sticky}) correspond to the
three mutually exclusive and exhaustive possibilities that $N\in A\cap B^{c},$ $N\in
A^{c}\cap B$ and $N\in A\cap B$ in (\ref{shuffle'}).

The same unit, coproduct and counit as before can be applied to make the sticky
shuffle product algebra into a Hopf algebra, but the definition of the antipode must be
modified \cite{Huds1} to
\begin{eqnarray}
&&\!\!\!\!\!\!\!\!(-1)^{m}S\left\{ L_{1}\otimes L_{2}\otimes \cdots \otimes L_{m}\right\} \notag\\
&=&\left\{ L_{m}\otimes L_{m-1}\otimes \cdots \otimes L_{1}\right\}
\,+\,\sum_{r=1}^{m}\ \sum_{1\leq k_{1}\,<k_{2}<\cdots <k_{r-1}<m} \\
&&\left\{ L_{k_{r-1}+1}L_{k_{r-1}+2}...L_{m}\otimes
L_{k_{r-2}+1}L_{k_{r-2}+2}...L_{k_{r-1}}\otimes \cdots \otimes
L_{1}L_{2}...L_{k_{1}}\right\} .\notag
\end{eqnarray}

The \emph{recovery formula }\cite{Bour}\emph{\ }expresses the homogeneous
components of an element $\alpha $ of $\mathcal{T}\left( \mathcal{L}\right) $ in terms of
the iterated coproduct $\Delta ^{(N)}\alpha$ by
\begin{equation}
\alpha _{N}=\left( \Delta ^{(N)}\alpha \right) _{(1,1,...,\overset{(N)}{1})}.
\label{recovery}
\end{equation}%
Here, $\Delta ^{(N)}$ is defined recursively by
\begin{equation}
\Delta ^{(2)}=\Delta\quad\text{and}\quad\Delta ^{(N)}=\left( \Delta \otimes \mbox{Id}_{\otimes ^{(N-2)}\mathcal{T}(\mathcal{L})}\right)
\circ \Delta ^{(N-1)}\quad\text{for}\quad N>2\,.
\end{equation}
Hence, it is a map from $\mathcal{T}\left( \mathcal{L} \right) $ to the $N$th tensor
power
\begin{equation}
\bigotimes\!^{(N)}\mathcal{T}(\mathcal{L})
\,=\,\bigotimes\!^{(N)}\bigoplus\limits_{n=0}^{\infty}\ \bigotimes\limits_{j=1}^{n}\mathcal{L}
\,=\bigoplus\limits_{n_{1},n_{2},...,n_{N}=0}^{\infty}\ \bigotimes_{r=1}^{N}\ \bigotimes\limits_{j_{r}=1}^{n_{r}}\mathcal{L}%
\end{equation}
so that $\Delta ^{(N)}\alpha $ has multirank components $\alpha _{\left(
n_{1},n_{2},...,n_{N}\right) }$ of all orders. The recovery formula (\ref%
{recovery}) also holds when $N=0$ and $N=1$ if we define $\Delta ^{(0)}$ and
$\Delta ^{(1)}$ to be the counit $\varepsilon $ and the identity map $\mbox{Id}_{%
\mathcal{T}\left( \mathcal{L}\right) }$ respectively.

Note that the recovery formula is the same for both the sticky and nonsticky cases; it
only involves the coproduct $\Delta $ which is one and the same map. However our
application of it will use the fact that\ $\Delta $ is multiplicative, $\Delta \left( \alpha \beta
\right) =\Delta \left( \alpha \right) \Delta \left( \beta \right)$, where the product on the
tensor square $\mathcal{T}\left( \mathcal{L}\right)\otimes\mathcal{T}\left(
\mathcal{L}\right)$ is defined by linear extension of the rule
\begin{equation}\label{eq.power}
(a\otimes a')(b\otimes b')\,=\,ab\otimes a'b'\,.
\end{equation}
This holds in particular with the sticky shuffle product as the product in
$\mathcal{T}\left( \mathcal{L}\right)$.

 \section{Moments and sticky shuffles.}

We now describe the connection between sticky shuffle products and iterated
stochastic integrals. We begin with the well-known fact that, for the
one-dimen\-sional Brownian motion $X$ and for $a\leq b,$
\begin{equation}
\left( X(b)-X(a\right) )^{2}\,=\,2\int_{a\leq x<b}\left( X(x)-X\left( a\right)
\right)dX(x)+\int_{a\leq x<b}dT(x),  \label{k}
\end{equation}%
where $T(x)=x$ is time. We introduce the \emph{It\^{o} algebra} $\mathcal{L=}%
\mathbb{C}\left\langle dX,dT\right\rangle $ of complex linear combinations of the basic
differentials $dX$ and $dT,$ which are multiplied according to the table
\begin{equation}
\begin{tabular}{c|c}
& $%
\begin{array}{cc}
dX & dT%
\end{array}%
$ \\ \hline
$%
\begin{array}{c}
dX \\
dT%
\end{array}%
$ & $%
\begin{array}{cc}
dT & 0 \\
0 & 0%
\end{array}%
$%
\end{tabular}%
,  \label{table}
\end{equation}%
together with the corresponding sticky shuffle Hopf algebra $\mathcal{T}%
\left( \mathcal{L}\right) .$ For each pair of real numbers $a<b,$ we introduce a map
$J_{a}^{b}$ from $\mathcal{T}\left( \mathcal{L}\right) $ to complex-valued random
variables on the probability space of the Brownian
motion $X$ by linear extension of the rule that, for arbitrary $%
dL_{1},dL_{2},\cdots dL_{m}\in \left\{ dX,dT\right\} $
\begin{eqnarray}
&&\!\!\!\!\!\!\!\!J_{a}^{b}\left\{ dL_{1}\otimes dL_{2}\otimes \cdots \otimes dL_{m}\right\}\notag
\\
&=&\int_{a\leq x_{1}<x_{2}<\dotsb
<x_{m}<b}dL_{1}(x_{1})\,dL_{2}(x_{2})\,dL_{3}(x_{3})\dotsm\,dL_{m}(x_{m}) \notag\\
&=&\int_{a}^{b}\!\!\dotsm\int_{a}^{x_{4}}\!\!\int_{a}^{x_{3}}\!\!\int_{a}^{x_{2}}
\!\!dL_{1}(x_{1})\,dL_{2}(x_{2})\,dL_{3}(x_{3})\dotsm\,\notag
dL_{m}(x_{m}) .
\end{eqnarray}
By convention $J_{a}^{b}$ maps the unit element of the algebra $\mathcal{T}%
\left( \mathcal{L}\right) $ to the unit random variable identically equal to 1.

Then (\ref{k}) can be restated as follows,
 \begin{equation}
J_{a}^{b}\left(\left\{
dX\right\} \right)J_{a}^{b}\left(\left\{
dX\right\} \right)\,=\,J_{a}^{b}\left(\left\{
dX\right\} \left\{
dX\right\}\right),
\end{equation}
using the fact that $\left\{ dX\right\} ^{2}=2\left\{ dX\otimes dX\right\} +\left\{ dT\right\} .$

The following more general Theorem is probably known to many classical and
quantum probabilists; the quantum version was first given in \cite{CoEH}.

\begin{theorem}\label{general}
For arbitrary $\alpha $ and $\beta $ in $\mathcal{T}\left( \mathcal{L}\right)$,
\begin{equation*} \label{j}
J_{a}^{b}(\alpha )J_{a}^{b}(\beta )\,=\,J_{a}^{b}(\alpha \beta )\,.
\end{equation*}
\end{theorem}

\begin{proof}
By bilinearity it is sufficient to consider the case when
\begin{equation}
\alpha =\left\{ dL_{1}\otimes dL_{2}\otimes \cdots \otimes dL_{m}\right\}
,\ \beta =\left\{ dL_{m+1}\otimes dL_{m+2}\otimes \cdots \otimes
dL_{m+n}\right\}
\end{equation}%
for $dL_{1},dL_{2},\cdots ,dL_{m+n}\in \left\{ dX,dT\right\} .$ In this case
Theorem\,\ref{general} follows, using the inductive definition (\ref{sticky}) for the
sticky shuffle product, from the product form of It\^{o}'s formula,%
\begin{equation}
d\left( \xi \eta \right) \,=\,\left( d\xi \right) \eta +\xi d\eta +\left( d\xi
\right) d\eta   \label{ito}
\end{equation}%
where stochastic differentials of the form $d\xi =FdX+GdT,$ with stochastically
integrable processes $F$ and $G$, are multiplied using table (\ref
{table}).\bigskip
\end{proof}

For planar Brownian motion $R=\left( X,Y\right) $ the Ito table (\ref{table}%
)~becomes%
\begin{equation}
\begin{tabular}{c|c}
& $%
\begin{array}{cc}
\begin{array}{cc}
dX & dY%
\end{array}
& dT%
\end{array}%
$ \\ \hline
$%
\begin{array}{c}
\begin{array}{c}
dX \\
dY%
\end{array}
\\
dT%
\end{array}%
$ & $%
\begin{array}{cc}
\begin{array}{cc}
dT & 0  \\
0 & dT%
\end{array}
&
\begin{array}{c}
0 \\
0%
\end{array}
\\
\begin{array}{cc}
0\ \ & 0%
\end{array}
& 0%
\end{array}%
$%
\end{tabular}%
.  \label{table'}
\end{equation}%
For the quantum planar Brownian motion $\left( P^{\left( \sigma \right)
},Q^{\left( \sigma \right) }\right) $ it becomes%
\begin{equation}
\begin{tabular}{c|c}
& $%
\begin{array}{cc}
\begin{array}{cc}
dP^{\left( \sigma \right) } & dQ^{\left( \sigma \right) }%
\end{array}
& dT%
\end{array}%
$ \\ \hline
$%
\begin{array}{c}
\begin{array}{c}
dP^{\left( \sigma \right) } \\
dQ^{\left( \sigma \right) }%
\end{array}
\\
dT%
\end{array}%
$ & $%
\begin{array}{cc}
\begin{array}{cc}
\sigma ^{2}dT & -idT \\
idT & \sigma ^{2}dT%
\end{array}
&
\begin{array}{c}
0 \\
0%
\end{array}
\\
\begin{array}{cc}
0\ \ \ &\ 0%
\end{array}
& 0%
\end{array}%
$%
\end{tabular}%
.  \label{table''}
\end{equation}

\begin{theorem}
Theorem \ref{general} holds when $\mathcal{L}$ is either of the algebras defined by
the multiplication tables (\ref{table'}) and (\ref{table''}).
\end{theorem}

\begin{remark}
In both cases this follows from the corresponding It\^{o} product rule (\ref{ito}). For
classical planar Brownian motion the It\^{o}
product rule is well-known. For the quantum case, when $\sigma =1$ see \cite%
{Part} and when $\sigma >1$ see \cite{HuLi}.
\end{remark}

In view of (\ref{classical}) and (\ref{quantum})
\begin{eqnarray}
\mathcal{A}_{[a,b)} &=&\dfrac{1}{2}J_{a}^{b}(dX\otimes dY-dY\otimes dX) \\
\mathcal{B}_{[a,b)}^{\left( \sigma \right) } &=&\dfrac{1}{2}J_{a}^{b}\left(
dP^{\left( \sigma \right) }\otimes dQ^{\left( \sigma \right) }-dQ^{\left(
\sigma \right) }\otimes dP^{\left( \sigma \right) }\right)
\end{eqnarray}

For use below we note that the table (\ref{table'}) becomes%
\begin{equation}
\begin{tabular}{c|c}
& $%
\begin{array}{cc}
\begin{array}{cc}
dZ & d\bar{Z}%
\end{array}
& dT%
\end{array}%
$ \\ \hline
$%
\begin{array}{c}
\begin{array}{c}
dZ \\
d\bar{Z}%
\end{array}
\\
dT%
\end{array}%
$ & $%
\begin{array}{cc}
\begin{array}{cc}
0 &
{\frac12}%
dT \\
{\frac12}%
dT & 0%
\end{array}
&
\begin{array}{c}
0 \\
0%
\end{array}
\\
\begin{array}{cc}
0\ \ & \ \ 0%
\end{array}
& 0%
\end{array}%
$%
\end{tabular}
\label{mod}
\end{equation}%
in terms of the basis $\left( dZ,d\bar{Z},dT\right) $ where
\begin{equation}
dZ=\frac{1}{2}\left( -idX+dY\right),\ \ d\bar{Z}=\frac{1}{2}\left(
idX+dY\right) .
\end{equation}%
Correspondingly%
\begin{equation}
\mathcal{A}_{[a,b)}=iJ_{a}^{b}\left( dZ\otimes d\bar{Z}-d\bar{Z}\otimes
dZ\right)   \label{class}
\end{equation}%
Similarly (\ref{table''}) becomes%
\begin{equation}
\begin{tabular}{c|c}
& $%
\begin{array}{cc}
\begin{array}{cc}
\ \ \ dA^{\left( \sigma \right) }& \ \ \ \ \ \ \ \ \ \ \ dA^{\dagger \left( \sigma \right) }%
\end{array}
&\ \ \ \ \ dT%
\end{array}%
$ \\ \hline
$%
\begin{array}{c}
\begin{array}{c}
dA^{\left( \sigma \right) } \\
dA^{\dagger \left( \sigma \right) }%
\end{array}
\\
dT%
\end{array}%
$ & $%
\begin{array}{cc}
\begin{array}{cc}
0 &
{\frac12}%
\left( \sigma ^{2}+1\right) dT \\
{\frac12}%
\left( \sigma ^{2}-1\right) dT & 0%
\end{array}
&
\begin{array}{c}
0 \\
0%
\end{array}
\\
\begin{array}{cc}
0\ \ \ \ \ \  \ \  &\ \ \ \ \ \ \ \ 0%
\end{array}
& 0%
\end{array}%
$%
\end{tabular}%
.  \label{mod'}
\end{equation}%
in terms of the basis $( dA^{(\sigma) },dA^{\dagger(\sigma) },dT) $ where
\begin{equation}
\text{ }dA^{\left( \sigma \right) }=%
{\frac12}%
\left( -idP^{\left( \sigma \right) }+dQ^{\left( \sigma \right) }\right)
,\ \ dA^{\dagger \left( \sigma \right) }=%
{\frac12}%
\left( idP^{\left( \sigma \right) }+dQ^{\left( \sigma \right) }\right) .
\label{basis}
\end{equation}

For the basis $(d\hat{P}^{(\sigma) },d\hat{Q}^{(\sigma) },dT) ,$ (\ref{table''}) becomes%
\begin{equation}
\begin{tabular}{c|c}
& $%
\begin{array}{cc}
\begin{array}{cc}
\ \ d\hat{P}^{(\sigma)} &\ \ d\hat{Q}^{(\sigma)}%
\end{array}
&\ \ dT%
\end{array}%
$ \\ \hline
$%
\begin{array}{c}
\begin{array}{c}
d\hat{P}^{\left( \sigma \right) } \\
d\hat{Q}^{\left( \sigma \right) }%
\end{array}
\\
dT%
\end{array}%
$ & $%
\begin{array}{cc}
\begin{array}{cc}
dT & -i\sigma ^{-2}dT \\
i\sigma ^{-2}dT & dT%
\end{array}
&
\begin{array}{c}
0 \\
0%
\end{array}
\\
\begin{array}{cc}
0\ \ \ \ &\ \ \ \ 0%
\end{array}
& 0%
\end{array}%
$%
\end{tabular}%
.  \label{mod''}
\end{equation}%
which deforms to the classical table (\ref{table'}) as $\sigma \rightarrow
\infty .$ Similarly, for the basis $(d\hat{A}^{(\sigma)},d%
\hat{A}^{\dagger(\sigma) },dT) $ where
\begin{equation}
\text{ }d\hat{A}^{\left( \sigma \right) }=%
{\frac12}%
\left( -i\,d\hat{P}^{\left( \sigma \right) }+d\hat{Q}^{\left( \sigma \right)
}\right) ,\ \ d\hat{A}^{\dagger \left( \sigma \right) }=%
{\frac12}%
\left( i\,d\hat{P}^{\left( \sigma \right) }+d\hat{Q}^{\left( \sigma \right)
}\right)   \label{hat}
\end{equation}%
we have%
\begin{equation}
\begin{tabular}{c|c}
& $%
\begin{array}{cc}
\begin{array}{cc}
\ \ \ \ d\hat{A}^{\left( \sigma \right) }\  &\ \ \ d\hat{A}^{\dagger \left( \sigma \right) }%
\end{array}
&\ \ dT%
\end{array}%
$ \\ \hline
$%
\begin{array}{c}
\begin{array}{c}
d\hat{A}^{\left( \sigma \right) } \\
d\hat{A}^{\dagger \left( \sigma \right) }%
\end{array}
\\
dT%
\end{array}%
$ & $%
\begin{array}{cc}
\begin{array}{cc}
\ \  0 &
\ \ \ \ \ %
\sigma_{\!+}\,dT \\
\ \sigma_{\!-}\,dT\  &\ \ \ 0%
\end{array}
&
\begin{array}{c}
\  \ 0 \\
\ 0%
\end{array}
\\
\begin{array}{cc}
  0\ \ \ \ \ \ &
\ \ \ \ \  0%
\end{array}
&\ 0%
\end{array}%
$%
\end{tabular}
\label{hat'}
\end{equation}%
with
\begin{equation}
\sigma _{\pm }=%
{\frac12}%
\left(1\pm \sigma ^{-2}\right)\,,
\end{equation}
which becomes isomorphic to (\ref{mod}) when $\sigma \rightarrow \infty .$
The normalized quantum L\'{e}vy area which is our main concern is%
\begin{equation}
\mathcal{\hat{B}}_{[a,b)}^{\left( \sigma \right) }\,=\,iJ_{a}^{b}\left( d\hat{A}%
^{\left( \sigma \right) }\otimes d\hat{A}^{\dagger \left( \sigma \right) }-d%
\hat{A}^{\dagger \left( \sigma \right) }\otimes d\hat{A}^{\left( \sigma
\right) }\right) .  \label{norm}
\end{equation}

In the following theorem, the basis referred to is any of those for
which the respective algebras have multiplication tables (\ref{table'}), (%
\ref{table''}), (\ref{mod}), (\ref{mod'}), (\ref{mod''}) or (\ref{hat'}).

\begin{theorem}\label{sz.LT}
For arbitrary $n\in \mathbb{N,}$ $a<b\in \mathbb{R}$ and basis elements $dL_{1}$,
$dL_{2}$, \dots, $dL_{n}$,
\begin{equation*}
\mathbb{E}\left[ J_{a}^{b}\left\{ dL_{1}\otimes dL_{2}\otimes ...\otimes
dL_{n}\right\} \right] =0
\end{equation*}%
unless%
\begin{equation*}
dL_{1}=dL_{2}=\cdots =dL_{n}=dT.
\end{equation*}
\end{theorem}

\begin{proof}
If $dL_{n}\neq dT$ then%
\begin{equation}
J_{a}^{b}\left\{ dL_{1}\otimes dL_{2}\otimes ...\otimes dL_{n}\right\}
\,=\,\int_{a}^{b}J_{a}^{x}\left\{ dL_{1}\otimes dL_{2}\otimes ...\otimes
dL_{n-1}\right\} dL_{n}(x).  \label{sto}
\end{equation}%
In the classical cases (\ref{table'}) and (\ref{mod}), $L_{n}$ is a real or complex-valued
martingale and the expectation of the stochastic integral against $dL_{n}$ vanishes.
When $\sigma =1$ it also vanishes in the cases (\ref{table''}), (\ref{mod'}) and
(\ref{mod''}) by the first
fundamental formula of quantum stochastic calculus in the Fock space $%
\mathcal{F}$ \cite{Part}. When $\sigma >1$ it vanishes as may be seen for
example by realising the processes $P^{\left( \sigma \right) }$ and $%
Q^{(\sigma )}$ in the tensor product of $\mathcal{F}$ with its Hilbert space
dual, $\mathcal{F\otimes \bar{F}}$ , equipped with the tensor product $%
e\left( 0\right) \otimes \overline{e\left( 0\right) }$ of the Fock vacuum vector with its dual
vector as
\begin{eqnarray}
P^{\left( \sigma \right) } &\,=\,&%
%
\sqrt{{\frac12}\left( \sigma ^{2}+1\right)} \left( P^{\left( \sigma\right) }\otimes \bar{I}%
\right) +%
%
\sqrt{{\frac12}\left( \sigma ^{2}-1\right)} \left( I\otimes \bar{P}^{(\sigma )}\right) \\
Q^{\left( \sigma \right) } &\,=\,&%
\sqrt{{\frac12}\left( \sigma ^{2}+1\right)} \left( Q^{\left( \sigma\right) }\otimes \bar{I}%
\right) +%
%
\sqrt{{\frac12}\left( \sigma ^{2}-1\right)} \left( I\otimes \bar{Q}^{(\sigma )}\right)
\end{eqnarray}%
and again invoking the first fundamental formula. Thus in all cases%
\begin{equation}
\mathbb{E}\left[ J_{a}^{b}\left\{ dL_{1}\otimes dL_{2}\otimes ...\otimes
dL_{n}\right\} \right] =0
\end{equation}%
unless $dL_{n}=dT.$

If $dL_{n}=dT$ then by Fubini's theorem we can write%
\begin{eqnarray}
&&\!\!\!\!\!\!\!\!\mathbb{E}\left[ J_{a}^{b}\left\{ dL_{1}\otimes dL_{2}\otimes ...\otimes
dL_{n}\right\} \right]  \notag\\
&=\,&\int_{a}^{b}\left\{ \mathbb{E}\left[ J_{a}^{x}\left\{ dL_{1}\otimes
dL_{2}\otimes ...\otimes dL_{n-1}\right\} \right] \right\} dx \notag \\
&=\,&0,
\end{eqnarray}%
unless $dL_{n-1}=dT$, by the previous argument. By repetition we see that
\begin{equation}
\mathbb{E}\left[ J_{a}^{b}\left\{ dL_{1}\otimes dL_{2}\otimes ...\otimes dL_{n}\right\}
\right] \,=\,0,
\end{equation}
unless each of $dL_{n},$ $dL_{n-1}$, $dL_{n-2}$, \dots, $dL_{1}$ is equal to
$dT$.\end{proof}

\bigskip

Now consider the moments sequence of the normalized quantum L\'{e}vy area of
variance $\sigma ^{2}$ in the form (\ref{norm}). In view of Theorem 3
\begin{eqnarray}
\left[ \mathcal{\hat{B}}_{[a,b)}^{\left( \sigma \right) }\right] ^{n}
&=\,&i^{n}\left( J_{a}^{b}\left( d\hat{A}^{\left( \sigma \right) }\otimes d%
\hat{A}^{\dagger \left( \sigma \right) }-d\hat{A}^{\dagger \left( \sigma
\right) }\otimes d\hat{A}^{\left( \sigma \right) }\right) \right) ^{n} \notag \\
&=\,&i^{n}J_{a}^{b}\left( \left\{ d\hat{A}^{\left( \sigma \right) }\otimes d%
\hat{A}^{\dagger \left( \sigma \right) }-d\hat{A}^{\dagger \left( \sigma
\right) }\otimes d\hat{A}^{\left( \sigma \right) }\right\} ^{n}\right)
\end{eqnarray}

The $n$th sticky shuffle power $\left\{ d\hat{A}^{\left( \sigma \right) }\otimes
d\hat{A}^{\dagger \left( \sigma \right) }-d\hat{A}^{\dagger \left( \sigma \right) }\otimes
d\hat{A}^{\left( \sigma \right) }\right\} ^{n}$ will consist of non-sticky shuffle products of
rank $2n$ together with terms of lower ranks $n,n+1,...,2n-1$,$~$all of which except
the rank $n$ term will
contain one or more copies of $d\hat{A}^{\left( \sigma \right) }$ and $d\hat{%
A}^{\dagger \left( \sigma \right) },$ and will thus not contribute to the expectation in view
of Theorem\,\ref{sz.LT}. The term of rank $n$ will be a multiple
of $dT\otimes dT\cdots \otimes \overset{(n)}{dT}.$ Thus we can write%
\begin{eqnarray}
&&\!\!\!\!\!\!\!\!\left\{ d\hat{A}^{\left( \sigma \right) }\otimes d\hat{A}^{\dagger \left(
\sigma \right) }-d\hat{A}^{\dagger \left( \sigma \right) }\otimes d\hat{A}%
^{\left( \sigma \right) }\right\} ^{n}  \notag \\
&=\,&w_{n}^{\left( \sigma \right) }\bigl\{ dT\otimes dT\cdots \otimes \overset{%
(n)}{dT}\bigr\} +\,\text{terms of rank}>n.  \label{rank}
\end{eqnarray}%
for some coefficient $w_{n}^{\left( \sigma \right) }.$ The corresponding moment is given
by
\begin{eqnarray}
\mathbb{E}\left[ \mathcal{\hat{B}}_{[a,b)}^{\left( \sigma \right) }\right]
^{n} &=\,&i^{n}w_{n}^{\left( \sigma \right) }\mathbb{E}\left[ J_{a}^{b}\left(
\left\{ dT\otimes dT\cdots \otimes \overset{(n)}{dT}\right\} \right) \right]
\notag \\
&=\,&i^{n}w_{n}^{\left( \sigma \right) }\int_{a\leq x_{1}<x_{2}<\cdots
<x_{n}<b}dx_{1}dx_{2}...dx_{n}  \notag \\
&=\,&i^{n}w_{n}^{\left( \sigma \right) }\frac{\left( b-a\right) ^{n}}{n!}.
\label{moment}
\end{eqnarray}

By the recovery formula (\ref{recovery}) and the multiplicativity of the $n$%
th order coproduct $\Delta ^{\left( n\right) },$%
\begin{eqnarray}
&&\!\!\!\!\!\!\!\!w_{n}^{\left( \sigma \right) }dT\otimes dT\cdots \otimes \overset{(n)}{dT} \notag\\
&=\,&\left\{\left\{ d\hat{A}^{\left( \sigma \right) }\otimes d\hat{A}^{\dagger
\left( \sigma \right) }-d\hat{A}^{\dagger \left( \sigma \right) }\otimes d%
\hat{A}^{\left( \sigma \right) }\right\} ^{n}\right\}_n \notag\\
&=\,&\left( \Delta ^{(n)}\left( \left\{ d\hat{A}^{\left( \sigma \right)
}\otimes d\hat{A}^{\dagger \left( \sigma \right) }-d\hat{A}^{\dagger \left(
\sigma \right) }\otimes d\hat{A}^{\left( \sigma \right) }\right\}
_{{}}^{n}\right) \right) _{\!\!(1,1,...,\overset{\left( n\right) }{1})} \\
&=\,&\left( \left( \Delta ^{(n)}\left( \left\{ d\hat{A}^{\left( \sigma \right)
}\otimes d\hat{A}^{\dagger \left( \sigma \right) }-d\hat{A}^{\dagger \left(
\sigma \right) }\otimes d\hat{A}^{\left( \sigma \right) }\right\} \right)
\right) ^{n}\right) _{\!\!(1,1,...,\overset{\left( n\right) }{1})}.\notag
\end{eqnarray}%
Now%
\begin{eqnarray}
&&\!\!\!\!\!\!\!\Delta ^{(n)}\left( \left\{ d\hat{A}^{\left( \sigma \right) }\otimes d\hat{%
A}^{\dagger \left( \sigma \right) }-d\hat{A}^{\dagger \left( \sigma \right)
}\otimes d\hat{A}^{\left( \sigma \right) }\right\} \right)  \notag\\
&=&\sum_{1\leq j\leq n}1_{\mathcal{T}\left( \mathcal{L}\right) }\otimes\cdots \otimes\overset{(j)}{\left\{ d\hat{A}^{\left( \sigma \right) }\otimes d\hat{A}^{\dagger \left(
\sigma \right) }-d\hat{A}^{\dagger \left( \sigma \right) }\otimes d\hat{A}%
^{\left( \sigma \right) }\right\}}
\otimes\cdots \otimes 1_{\mathcal{T}\left( \mathcal{L}\right) }\notag\\
&&+\sum\limits_{1\leq j\,<k\leq n}\biggl( 1_{\mathcal{T}\left( \mathcal{L}%
\right) }\otimes \cdots \otimes \overset{\left( j\right) }{\left\{ d\hat{A}%
^{\left( \sigma \right) }\right\} }\otimes \cdots \otimes \overset{\left(
k\right) }{\left\{ d\hat{A}^{\dagger \left( \sigma \right) }\right\} }%
\otimes \cdots \otimes \overset{\left( n\right) }{1_{\mathcal{T}\left(
\mathcal{L}\right) }} \notag \\
&& -1_{\mathcal{T}\left( \mathcal{L}\right) }\otimes \cdots \otimes
\overset{\left( j\right) }{\left\{ d\hat{A}^{\dagger \left( \sigma \right)
}\right\} }\otimes \cdots \otimes \overset{\left( k\right) }{\left\{ d\hat{A}%
^{\left( \sigma \right) }\right\} }\otimes \cdots \otimes \overset{\left(
n\right) }{1_{\mathcal{T}\left( \mathcal{L}\right) }}\biggr).
\end{eqnarray}%
The first term of this sum, being of rank $2,$ cannot contribute to the component of joint rank $(1,1,...,
\overset{\left( n\right) }{1})$ of the $n$th power of $\Delta ^{(n)}\left( \left\{ dX\otimes
dY-dY\otimes dX\right\} \right)$, where product in the n\textit{th} tensor power
$\bigotimes\!^{(N)}\mathcal{T}(\mathcal{L})$ is defined exactly as in the case $n=2$ in 
\eqref{eq.power}. Thus
\begin{eqnarray}\label{eq.111}
&&\!\!\!\!\!\!\!\!w_{n}^{\left( \sigma \right) }dT\otimes dT\cdots \otimes \overset{(n)}{dT} \notag\\
&=\,&\left( \left( \Delta ^{(n)}\left( \left\{ d\hat{A}^{\left( \sigma \right)
}\otimes d\hat{A}^{\dagger \left( \sigma \right) }-d\hat{A}^{\dagger \left(
\sigma \right) }\otimes d\hat{A}^{\left( \sigma \right) }\right\} \right)
\right) ^{\!n\,}\right) _{\!\!(1,1,...,\overset{\left( n\right) }{1})} \\
&=\,&\left( \left( \sum\limits_{1\leq j\,<k\leq n}\biggl( 1_{\mathcal{T}\left(
\mathcal{L}\right) }\otimes \cdots \otimes \overset{\left( j\right) }{%
\left\{ d\hat{A}^{\left( \sigma \right) }\right\} }\otimes \cdots \otimes
\overset{\left( k\right) }{\left\{ d\hat{A}^{\dagger \left( \sigma \right)
}\right\} }\otimes \cdots \otimes \overset{\left( n\right) }{1_{\mathcal{T}%
\left( \mathcal{L}\right) }}\biggr. \right. \right. \notag \\
&&\left. \left. \biggl. -1_{\mathcal{T}\left( \mathcal{L}\right) }\otimes
\cdots \otimes \overset{\left( j\right) }{\left\{ d\hat{A}^{\dagger \left(
\sigma \right) }\right\} }\otimes \cdots \otimes \overset{\left( k\right) }{%
\left\{ d\hat{A}^{\left( \sigma \right) }\right\} }\otimes \cdots \otimes
\overset{\left( n\right) }{1_{\mathcal{T}\left( \mathcal{L}\right) }}\biggr)
\right)^{\!\!\!n\,}\right) _{\!\!\!(1,1,...,\overset{\left( n\right) }{1})}\notag
\end{eqnarray}%
This calculation of $w_{n}^{\left( \sigma \right) }$ can be finished using some
combinatorics. We do that in the following two sections.

\section{Some background about Eulerian and Euler numbers.}

In this section, we present some simple lemmas about Euler numbers, Eulerian
numbers and Euler polynomials. They are of sufficient general nature to be of
potential interest elsewhere. Many similar results and basics 
can be found in \cite{Pete} and \cite{Stan}.

A permutation $\mathfrak{s}$ in the symmetric group $\cS_{n}$ is a \emph{zigzag
permutation} (misleadingly also called alternating permutation)
if $\mathfrak{s}(1)>\mathfrak{s}(2)<\mathfrak{s}(3)>\mathfrak{s}(4)<\dotsb $%
. In other words, $\mathfrak{s}$ is zigzag if $\mathfrak{s}(1)>\mathfrak{s}(2)$ and
\begin{equation}
\text{either}\quad\mathfrak{s}(j-1)<\mathfrak{s}(j)>\mathfrak{s}(j+1)\quad\text{or}\quad\mathfrak{s}(j-1)>\mathfrak{s}(j)<\mathfrak{s}(j+1)  \label{1}
\end{equation}%
for all $j\in \{2,3\dotsc ,n-1\}$.
If we have the initial condition $\mathfrak{s}(1)<\mathfrak{s}(2)$, instead of $\mathfrak{s}(1)>%
\mathfrak{s}(2)$, we may call $\mathfrak{s}$ \emph{zagzig}. The number of all zigzag
permutations in $\cS_{n}$ is called the \emph{Euler zigzag number} $A_{n}$. These
numbers occur in many places, for example, as the coefficients of $\frac{z^{2n}}{(2n)!}$
in the Maclaurin series of $\sec
(z)+\tan (z)$. In this paper, we will meet them as the number of \emph{%
forth-back permutations}, as we call them. These are the permutations $%
\mathfrak{s}\in \cS_{n}$ with
\begin{equation}
\text{either}\quad\mathfrak{s}^{-1}(j)<j>\mathfrak{s}(j)\quad\text{or}\quad\mathfrak{s}^{-1}(j)>j<\mathfrak{s}(j)  \label{2}
\end{equation}%
for all $j\in \{1,2,\dotsc ,n\}$. Since no forth-back permutation can contain a cycle of odd
length, $n$ must be even for there to exist forth-back permutations, say $n=2m>0$. In
that case, we have the following lemma:

\begin{lemma}
\label{lem.fb} The number of forth-back permutations in $\cS_{2m}$ is the Euler zigzag
number $A_{2m}$.
\end{lemma}

\begin{proof}
A bijection between the forth-back permutations $\mathfrak{s}$ and the zigzag
permutations in $\cS_{2m}$ is obtained by applying the so-called \emph{transformation
fundamentale} \cite{FoSc}. To perform this transformation, we write $\mathfrak{s}$ in
cycle notation
\begin{eqnarray}
\mathfrak{s} &\,=\,&(\mathfrak{s}_{1},\mathfrak{s}_{2},\dotsc ,\mathfrak{s}%
_{\ell _{2}-1})(\mathfrak{s}_{\ell _{2}},\mathfrak{s}_{\ell _{2}+1},\dotsc ,%
\mathfrak{s}_{\ell _{3}-1})(\mathfrak{s}_{\ell _{3}},\mathfrak{s}_{\ell
_{3}+1},\dotsc ,\mathfrak{s}_{\ell _{4}-1})\dotsm  \notag \\
&&(\mathfrak{s}_{\ell _{m}},\mathfrak{s}_{\ell _{m}+1},\dotsc ,\mathfrak{s}%
_{2m})\,.  \label{3}
\end{eqnarray}%
This representation and the numbers \ $\mathfrak{s}_{j}$ are uniquely determined if we
require that the first entry of every cycle is bigger than
all other entries in that cycle, and also that $\mathfrak{s}_{1}<\mathfrak{s}%
_{\ell _{2}}<\mathfrak{s}_{\ell _{3}}<\dotsb <\mathfrak{s}_{\ell _{m}}$. The new
permutation $\mathfrak{\bar{s}}$ is then obtained by forgetting brackets and setting
$\mathfrak{\bar{s}}(j):=\mathfrak{s}_{j}$. We just have to see that this actually yields a
bijection $\mathfrak{s}\mapsto \mathfrak{\bar{s}} $ between forth-back and zigzag
permutations. To do this we procede as follows.

Assume first that $\mathfrak{s}$ is forth-back. Then all cycles necessarily have even
length and the permutation $\mathfrak{\bar{s}}$ is obviously
zigzag, $\mathfrak{s}_{1}>\mathfrak{s}_{2}<\mathfrak{s}_{3}>\mathfrak{s}%
_{4}<\dotsb >\mathfrak{s}_{2m}$. Conversely, let us show that every zigzag
permutation $\bar{\mathfrak{s}}$ has a unique pre-image $\mathfrak{s}$, and
that that pre-image is forth-back.  To construct a pre-image $\mathfrak{s}$ of $%
\bar{\mathfrak{s}}$, we only need to find suitable numbers $\ell _{j}$,
which indicate where we have to insert brackets into the sequence $(%
\mathfrak{s}_{1},\mathfrak{s}_{2},\dotsc ,\mathfrak{s}_{2m}):=(\mathfrak{s}%
(1),\mathfrak{s}(2),\dotsc ,\mathfrak{s}(2m))$ to actually get a pre-image. However, if
we have already found $\ell _{2},\ell _{3},\dotsc ,\ell _{j}$,
then $\ell _{j+1}$ is necessarily the first index $x$ with $\mathfrak{s}_{x}>%
\mathfrak{s}_{\ell _{j}}$. Using this, we can construct a pre-image $%
\mathfrak{s}$ of $\bar{\mathfrak{s}}$ in $\cS_{2m}$, and it is uniquely determined.
Moreover, if $\bar{\mathfrak{s}}$ is zigzag then this construction ensures that
$\mathfrak{s}_{1}$ and the $\mathfrak{s}_{\ell _{j}}$ are peaks and their neighbors and
$\mathfrak{s}_{2m}$ are valleys.
Since also $\mathfrak{s}_{1}>\mathfrak{s}_{\ell _{2}-1}$, $\mathfrak{s}%
_{\ell _{2}}>\mathfrak{s}_{\ell _{3}-1}$, \dots , $\mathfrak{s}_{\ell _{m}}>%
\mathfrak{s}_{2m}$, insertion of brackets before the peaks $\ell _{j}$ yields forth-back
cycles in $\mathfrak{s}$.

With the bijection established, it is now clear that there are as many
forth-back permutations as there are zigzag permutations in $\cS%
_{2m} $. This number is the Euler zigzag number $A_{2m}$.
\end{proof}

The number of forth-back permutations with just one cycle is given by the following
lemma. If $\cC_n$ denotes the subset of cyclic permutations in $\cS_n$, we have:

\begin{lemma}
\label{lem.cfb} The number of forth-back permutations in $\cC_{2m}$ is $A_{2m-1}$.
\end{lemma}

\begin{proof}
The cycle notation $\mathfrak{s}=(\mathfrak{s}_{1},\mathfrak{s}_{2},\dotsc ,%
\mathfrak{s}_{2m})$ of cyclic permutations $\mathfrak{s}\in \cS_{2m}$ is not uniquely
determined, as one may rotate the entries cyclically. It becomes uniquely determined if
we require that $\mathfrak{s}_{2m}=2m$. In
this case, removal of the last entry yields a sequence $(\mathfrak{s}_{1},%
\mathfrak{s}_{2},\dotsc ,\mathfrak{s}_{2m-1})$ that is zagzig (with $%
\mathfrak{s}_{1}<\mathfrak{s}_{2}$ as $\mathfrak{s}_{2m}$ was the biggest
entry of $\mathfrak{s}$). If we define $\bar{\mathfrak{s}}\in \cS%
_{2m-1}$ by setting $\bar{\mathfrak{s}}(j):=\mathfrak{s}_{j}$, for $%
j=1,2,\dotsc ,2m-1$, we obtain a bijection $\mathfrak{s}\mapsto \bar{%
\mathfrak{s}}$ from the cyclic forth-back permutations in $\cS_{2m}$ to the zagzig
permutations in $\cS_{2m-1}$. Indeed, every zagzig
permutation $\bar{\mathfrak{s}}$ in $\cS_{2m-1}$ has the cycle $%
\mathfrak{s}:=(\bar{\mathfrak{s}}(1),\bar{\mathfrak{s}}(2),\dotsc ,\bar{%
\mathfrak{s}}(2m-1),2m)$ as unique pre-image. The existence of this
bijection shows that the number of cyclic forth-back permutations in $%
\cS_{2m}$ is equal to the number of zagzig permutations in $\mathcal{%
S}_{2m-1}$, which is $A_{2m-1}$, as for zigzag permutations.
\end{proof}

This enumerative result about cyclic forth-back permutations can also be applied to
forth-back permutation with $k$ cycles of lengths $2m_{1},2m_{2},\dotsc ,2m_{k}$
(necessarily all even). To formulate this, we denote with $\cC_{n_{1},n_{2},\dotsc
,n_{k}}$ the set of all permutations in $\cS_{n}$ with $k$ cycles of lengths
$n_{1},n_{2},\dotsc ,n_{k}$, i.e.\ the set of permutations of \emph{typ}
$(n_1,n_2,\dotsc,n_k)$, as we say. We also denote with $\bbinom{n}{n_{1},n_{2},\dotsc
,n_{k}}$ the number of unordered partitions $\{N_1,N_2,\dotsc,N_k\}$ of the set
$\{1,2,\dotsc,n\}$ with $k$ blocks $N_j$ of sizes $|N_j|=n_j>0$. With this we get the
following more general formula:

\begin{lemma}
\label{lem.fbd} If positive integers $m_{1}\leq m_{2}\leq \dotsb \leq m_{k}$ with
$m_{1}+m_{2}+\dotsc +m_{k}=m$ are given, then the number of forth-back
permutations in $\cS_{2m}$ with $k$ cycles of lengths $2m_{1},2m_{2},\dotsc ,2m_{k}$
is
\begin{equation*}
\bigl|\{\cs\in\cC_{2m_{1},2m_{2},\dotsc ,2m_{k}}\,|\,\text{$\cs$ is forth-back}\}\bigr|\,=\,
\bbinom{2m}{2m_{1},2m_{2},\dotsc ,2m_{k}}\prod_{j=1}^{k}A_{2m_{j}-1}\,.
\end{equation*}%
In particular,%
\begin{equation*}
A_{2m}\,=\,\sum \bbinom{2m}{2m_{1},2m_{2},\dotsc ,2m_{k}}%
\prod_{j=1}^{k}A_{2m_{j}-1}\,,
\end{equation*}%
where the sum runs over all partitions $m_{1}+m_{2}+\dotsb +m_{k}$ of $m$, that is
over all non-decreasing sequences $m_{1}\leq m_{2}\leq \dotsb \leq
m_{k}$ of positive integers of every possible length $k$ with $%
m_{1}+m_{2}+\dotsb +m_{k}=m>0$.
\end{lemma}

\begin{proof}
If a partition $m_{1}+m_{2}+\dotsb +m_{k}=m$ of the number $m$ is given, then there are $%
\bbinom{2m}{2m_{1},2m_{2},\dotsc ,2m_{k}}$ partitions of the set $%
\{1,2,\dotsc ,2m\}$ into a set of $k$ blocks $N_j$ with sizes $2m_{j}$, $%
j=1,2,\dotsc ,k$. The block $N_j$ can be turned into a cyclic
forth-back permutation in $A_{2m_{j}-1}$ many ways, by Lemma\thinspace \ref%
{lem.cfb}. Hence, we get the stated expression for the number of forth-back
permutations of that type.

Moreover, it is easy to see that the sum over all possible expressions of this form gives
the number of all forth-back permutations, which is $A_{2m}$
by Lemma\thinspace \ref{lem.fb}. Indeed, every forth-back permutations $%
\mathfrak{s}$ in $\cS_{2m}$, has a certain number $k$ of cycles and a certain type,
certain lengths $2m_{1},2m_{2},\dotsc ,2m_{k}$ of its cycles. In this respect, every
partition $m_{1}+m_{2}+\dotsb +m_{k}=m$ is possible. Hence, the sum covers all
$A_{2m}$ forth-back permutations, as stated.
\end{proof}


In our investigations, we will also need to look at a certain notion of \emph{sign},
denoted by $\sn(\mathfrak{s})$, for permutations $\mathfrak{s}\in\cS_{n}$, defined by
\begin{equation}
\sn(\mathfrak{s})\,:=\,\prod_{j=1}^{n}(-1)^{\des(j,\mathfrak{s}%
(j))}\,,  \label{45+}
\end{equation}%
where
\begin{equation}
\des(h,k)\,:=\,%
\begin{cases}
0 & \text{if $h\leq k$}, \\
1 & \text{if $h>k$}.%
\end{cases}
\label{41+}
\end{equation}
We want to show that $\sum_{\mathfrak{s}\in\cS_{2m}^{\neq}}\sn(\mathfrak{s})=
(-1)^{m}A_{2m}$, where $\cS_{n}^{\neq}$ denotes the set of fixed-point-free
permutations in $\cS_{n}$. To establish this and similar results, we need to introduce
certain equivalence classes of permutations which are based on the notion of a transit
of a permutation. We call $h\in \left\{1,2,...,n\right\} $ a \emph{transit}
of the permutation $\mathfrak{s\in }\cS_{n}$ if
\begin{equation}
\text{either}\quad\mathfrak{s}^{-1}(h)<h<\mathfrak{s}(h)\quad\text{or}\quad
\mathfrak{s}^{-1}(h)>h>\mathfrak{s}(h)\ .  \label{51}
\end{equation}
Let $\cS_n^T$ denote the set of permutations in $\cS_n$ which contain a transit, and
let $\cC_{n_{1},n_{2},\dotsc ,n_{k}}^T$ be the set of permutations in
$\cC_{n_{1},n_{2},\dotsc,n_{k}}$ which contain a transit. Every permutation
$\mathfrak{s}$ with transit 
contains a unique smallest transit $h$, say inside a cycle $(j_1,j_2,\dotsc,\cs^{-1}(h),
h,\cs(h),\dotsc,j_{\ell})$ of length $\ell$, which we may also write as
\begin{equation}\label{eq.lc}
j_1\mto j_2\mto\dotsb\mto\cs^{-1}(h)\mto h\mto\cs(h)\mto\dotsm\mto j_{\ell}\mto j_1\,.
\end{equation}
We obtain a permutation $\cs'$ of $\{1,2,\dotsc,n\}\!\sm\!\{h\}$ by replacing the chain of
assignments $\cs^{-1}(h)\mto h\mto\cs(h)$ with the shorter chain
$\cs^{-1}(h)\lmto\cs(h)$. Hence, the new permutation $\cs'$ contains the cycle
\begin{equation}\label{eq.sc}
j_1\mto j_2\mto\dotsb\mto\cs^{-1}(h)\lmto\cs(h)\mto\dotsm\mto j_{\ell}\mto j_1\,.
\end{equation}
We define an equivalence relation $\sim$ on the set $\cS_n^T$. For two permutations
$\cs$ and $\crr$ with transit, we write $\cs\sim\crr$ if and only if $\cs$ and $\crr$ have
the same smallest transit $h$, if $\cs'=\crr'$ and if the smallest transit $h$ is missing
from the same cycle in $\cs'$ as in $\crr'$. The equivalence class of $\cs$ is denoted
as $[\cs]$. The equivalence relation $\sim$ can also be restricted to the sets of the form
$\cC_{n_{1},n_{2},\dotsc ,n_{k}}^T$. We have
$[\cs]\sb\cC_{n_{1},n_{2},\dotsc,n_{k}}^T$ whenever
$\cs\in\cC_{n_{1},n_{2},\dotsc,n_{k}}^T$.

\begin{example}
In the $8$-cycle $\cs:=(4,1,8,2,6,7,\mathbf{5},3)$ the number $5$ is the smallest
transit, a \emph{downwards transit} in this case, as $7>5>3$. If we remove it from the
cycle, and reinsert the $5$ as a transit in all possible ways into the remaining $7$-cycle
$(4,1,8,2,6,7,3)$, we get four permutations. The $5$ would not be a transit between
$4$ and $1$, but can be inserted between $1$ and $8$, yielding
$(4,1,\mathbf{5},8,2,6,7,3)$. Similarly, we also obtain $(4,1,8,\mathbf{5},2,6,7,3)$,
$(4,1,8,2,\mathbf{5},6,7,3)$ and the original permutation $(4,1,8,2,6,7,\mathbf{5},3)$.
These four $8$-cycles form the equivalence class $[\cs]$ of $\cs$ with respect to
$\sim$. Interestingly, two of the four $8$-cycles contain the number $5$ as upwards
transit, and their sign is opposite to the sign of the other two $8$-cycles with $5$ as
downwards transit, as one can easily check.
The situation is illustrated in Fig.\,1.
\begin{figure}[t]
\begin{center}
\includegraphics[
scale = .1]{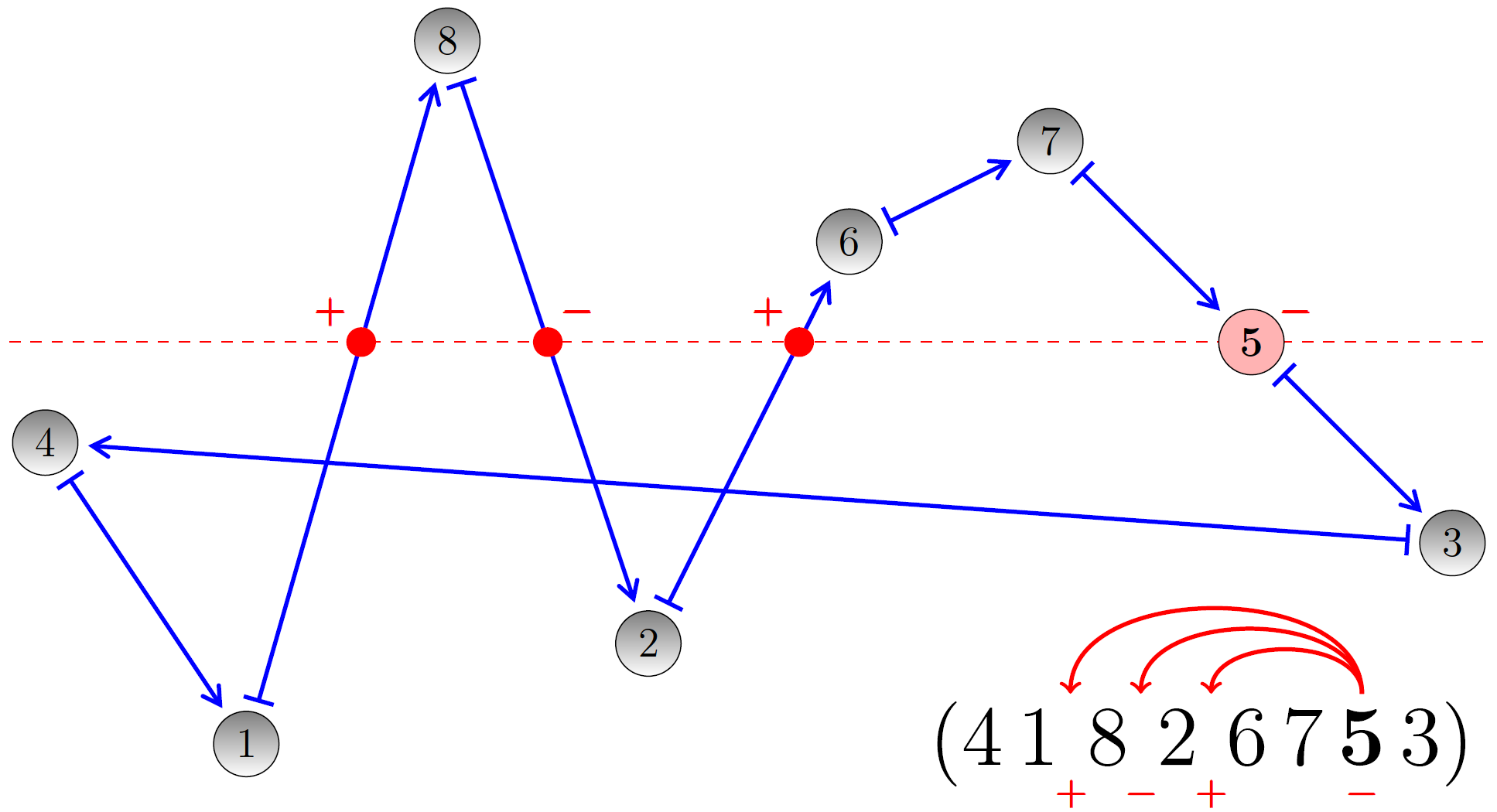}\label{fig.alt}
\vspace{-1em}
\caption{The cycle $(4,1,8,2,6,7,5,3)$ with smallest transit $h=5$.}
\end{center}
\end{figure}
\end{example}

The observation about the sign of the elements in equivalence classes that we made in
this example is no coincidence. We have the following lemma:

\begin{lemma}
\label{lem.eqc} If a permutation $\mathfrak{s}$ contains a transit, then there is an even
number of elements in its equivalence class $[\cs]$. One half of them have negative
sign, and one half have positive sign.
\end{lemma}

\begin{proof}
Assume the cycle that contains the smallest transit $h$ of $\cs$ is denoted as in
Equation\,\eqref{eq.lc}. If we walk once around the shortened cycle of $\cs'$ in
Equation\,\eqref{eq.sc} and observe the indices $j_1,j_2,j_3,\dotsc$ as a kind of
altitude, then we will cross the altitude $h$ as many times upwards (from below $h$ to
above $h$) as downwards (from above $h$ to
below $h$), see Fig.\,1. 
Hence, there are as many ways to reinsert $h$ as an upwards transit and as a
downwards transit. Therefore, one half of the permutations that we obtain have positive
sign, and one half have negative sign. The equivalence class $[\cs]$ of $\cs$ is as
claimed.
\end{proof}

With the help of Lemma\,\ref{lem.eqc}, we can now prove the following theorem:

\begin{theorem}
\label{sz.snd} Let a partition $n_{1}+n_{2}+\dotsc +n_{k}=n$ 
with $2\leq n_{1}\leq n_{2}\leq \dotsb \leq n_{k}$ be given. Then,
\begin{equation*}
\sum_{\mathfrak{s}\in\cC_{n_{1},n_{2},\dotsc ,n_{k}}\!\!\!\!\!\!\!\!\!\!\!\!}\!\!\!\!\sn(\mathfrak{s})
\,=\,
\begin{cases}
(-1)^{\frac{n}{2}}\bbinom{n}{n_{1},n_{2},\dotsc,n_{k}}\prod_{j=1}^{k}A_{n_{j}-1} & \text{if all $n_j$ are even,}\\
\quad0 & \text{otherwise.}%
\end{cases}
\end{equation*}%
In particular,%
\begin{equation*}
\sum_{\mathfrak{s}\in\cS_{n}^{\neq}}\sn(\mathfrak{s})\,=\,
\begin{cases}
(-1)^{\frac{n}{2}}A_{n} & \text{if $n$ is even,}\\
\quad0 & \text{otherwise.}%
\end{cases}
\end{equation*}%
\end{theorem}

\begin{proof}
We observe that we can cancel out all permutations in $\cC_{n_{1},n_{2},\dotsc
,n_{k}}$ that contain a transit, that is, all elements of $\cC_{n_{1},n_{2},\dotsc
,n_{k}}^T$. In fact, $\cC_{n_{1},n_{2},\dotsc ,n_{k}}^T$ is partitioned into equivalence
classes, and each of them cancels out by Lemma\,\ref{lem.eqc}. The remaining
elements of $\cC_{n_{1},n_{2},\dotsc ,n_{k}}$ do not contain a transit. Hence, if there
are any remaining permutations in $\cC_{n_{1},n_{2},\dotsc
,n_{k}}\!\sm\cC_{n_{1},n_{2},\dotsc ,n_{k}}^T$, they must be forth-back permutations. In
particular, in this case, all $n_j$ must necessarily be even. Now, Lemma\,\ref{lem.fbd}
yields the first stated result, since all forth-back permutations in $\cS_n^{\neq}$ have
sign $(-1)^{\frac{n}{2}}$. The second formula follows from this and the second formula
in Lemma\,\ref{lem.fbd}, but it can also be deduced from Lemma\,\ref{lem.fb} directly,
after canceling out equivalence classes in $\cS_n^{\neq}$.
\end{proof}

For the number of partitions in some of the previous results, we also have the following
formula:

\begin{lemma}
\label{lem.prt}
\begin{equation*}
\bbinom{n}{n_{1},n_{2},\dotsc ,n_{k}}=\binom{n}{n_{1},n_{2},\dotsc
,n_{k}}\frac{1}{k_{1}!k_{2}!\dotsm k_{r}!}=\frac{n!}{n_{1}!n_{2}!\dotsm
n_{k}!\,k_{1}!k_{2}!\dotsc k_{r}!}\,,
\end{equation*}%
where $k_{1},k_{2},\dotsc ,k_{r}$ are the multiplicities of the different elements in the
multi-set $\{n_{1},n_{2},\dotsc ,n_{k}\}$. (For example, the elements in the multi-set
$\{2,2,2,4,4\}$ have the multiplicities $k_{1}=3$ and $k_{2}=2$, and $r=2$.)
\end{lemma}

\begin{proof}
Without loss of generality, we may assume that $n_{1}\leq n_{2}\leq \dotsb
\leq n_{k}$. There are $\binom{n}{n_{1},n_{2},\dotsc ,n_{k}}:=\frac{n!}{%
n_{1}!n_{2}!\dotsm n_{k}!}$ ordered partitions (sequences of blocks) with block sizes
$n_{1},n_{2},\dotsc ,n_{k}$ (in that order). This number
can also be generated by first choosing all $\bbinom{n}{%
n_{1},n_{2},\dotsc ,n_{k}}$ unordered partitions (sets of blocks) with block sizes
$n_{1},n_{2},\dotsc ,n_{k}$, and then arranging each of them in all possible ways as a
sequences of blocks, i.e.\ as ordered partition. Hence, for each unordered partition
$\{N_1,N_2,\dotsc,N_k\}$, we have to see how many ways there are to arrange its
blocks in a sequence with nondecreasing cardinalities (equal to the sequence
$n_{1}\leq n_{2}\leq \dotsb \leq n_{k}$). Ambiguities in this order of the blocks are only
given for blocks of equal size, which correspond to multiplicities of the elements in the
multi-set $\{n_{1},n_{2},\dotsc ,n_{k}\}$. Hence, the number of ways is always
$k_{1}!k_{2}!\dotsc k_{r}!$, where $k_{1},k_{2},\dotsc ,k_{r}$ are
the multiplicities of the different elements in the multi-set $%
\{n_{1},n_{2},\dotsc ,n_{k}\}$. Combining this factor with the number of unordered
partitions yields the relation
\begin{equation}
\bbinom{n}{n_{1},n_{2},\dotsc ,n_{k}}k_{1}!k_{2}!\dotsm k_{r}!\,=\,%
\binom{n}{n_{1},n_{2},\dotsc ,n_{k}}\,,  \label{4}
\end{equation}%
which proves the lemma.
\end{proof}

In this paper we will also consider the number of \emph{descends} of sequences $%
(j_{1},j_{2},\dotsc ,j_{n})$ of $n\geq 1$ integers, that is, the number
\begin{equation}
\des(j_{1},j_{2},\dotsc ,j_{n})\,:=\,\left\vert \left\{ \,\ell \in
\{1,2,\dotsc ,n-1\}\,|\,j_{\ell }>j_{\ell +1}\right\} \right\vert \,,
\label{eq.des}
\end{equation}%
which generalizes $\des(h,k)\in\{0,1\}$ in \eqref{41+}. The number of permutations
$\mathfrak{s}\in \cS_{n}$ for which
\begin{equation}
\des(\mathfrak{s}(1),\mathfrak{s}(2),\dotsc ,\mathfrak{s}(n))=j
\end{equation}%
is the so-called \emph{Eulerian number} $\ebinom{n}{j}$. We follow \cite{Pete} in taking
this as the definition of the Eulerian numbers, but Eulerian numbers also count various
other kinds of objects, see \cite{Pete}. The corresponding generating function is the
so-called \emph{Euler polynomial}
\begin{equation}
S_{n}(\tau )\,:=\,\sum_{\mathfrak{s}\in \cS_{n}}\tau ^{\des(%
\mathfrak{s}(1),\mathfrak{s}(2),\dotsc ,\mathfrak{s}(n))}\,=\,%
\sum_{j=0}^{n-1}\ebinom{n}{j}\tau ^{j}\,.  \label{34}
\end{equation}%
In this paper, we will also need the closely related number of \emph{cyclic} descends,
defined by
\begin{equation}
\cdes(j_{1},j_{2},\dotsc ,j_{n})\,:=\,\des(j_{1},j_{2},\dotsc
,j_{n},j_{1})\,,  \label{35}
\end{equation}%
which has the following statistic (see also \cite{Pete}):

\begin{lemma}
\label{lem.cdes} All permutations $\mathfrak{s}\in \cS_{n}$ have
$0<\cdes(\mathfrak{s}(1),\mathfrak{s}(2),\dotsc ,\mathfrak{s}(n))<n$. For $0<j<n$, the
number of permutations $\mathfrak{s}\in \cS_{n}$ with exactly $j$ cyclic descents is
\begin{equation*}
\bigl|\bigl\{\mathfrak{s}\in \cS_{n}\,|\,\cdes(\mathfrak{s}(1),\mathfrak{s}(2),\dotsc
,\mathfrak{s}(n))=j\bigr\}\bigr|\,=\,n\ebinom{n-1}{j-1}\,.
\end{equation*}
In particular,
\begin{equation}\label{xxy}
\sum_{\mathfrak{s}\in \cS_{n}}\tau ^{\cdes(\mathfrak{s}(1),%
\mathfrak{s}(2),\dotsc ,\mathfrak{s}(n))}\,=\,n\,\tau S_{n-1}(\tau )\,.
\end{equation}
\end{lemma}

\begin{proof} Assume $\mathfrak{s}\in \cS_{n}$ with
$\cdes(\mathfrak{s}(1),\mathfrak{s}(2),\dotsc ,\mathfrak{s}(n))=j$. Since every
sequence of distinct integers has at least one cyclic descent and one cyclic ascent, $j$
cannot be $0$ or $n$, and $0<j<n$ as claimed. Now, let $M$ be such that
$\mathfrak{s}(M)$ is the biggest entry of the sequence $(\mathfrak{s}(1),\mathfrak{s}(2),\dotsc ,%
\mathfrak{s}(n))$. We construct a new shorter sequence
\begin{equation}
\bar{\mathfrak{s}}\,:=\,(\mathfrak{s}(M{+}1),\mathfrak{s}(M{+}2),\dotsc ,\mathfrak{s}(n),\mathfrak{s}%
(1),\mathfrak{s}(2),\dotsc ,\mathfrak{s}(M{-}1))
\end{equation} 
by removing $\mathfrak{s}%
(M)$ and gluing together the remaining halves in opposite order. Obviously, $%
\bar{\mathfrak{s}}$ has exactly $j-1$ descends. If we first rotate the
entries of the sequence $(\mathfrak{s}(1),\mathfrak{s}(2),\dotsc ,\mathfrak{s%
}(n))$ and then remove the biggest entry, we still get the same sequence $%
\bar{\mathfrak{s}}$ in the same way. This idea shows that removal of the
biggest entry yields an $n$ to $1$ correspondence $\mathfrak{s}\mapsto \bar{%
\mathfrak{s}}$ between the permutations in $\cS_{n}$ with $j$ cyclic descends and the
permutations in $\cS_{n-1}$ with $j-1$ descends. Hence, there are
$n\ebinom{n-1}{j-1}$ permutations in $\cS_{n}$ with exactly $j$ cyclic descends. In
particular, this number is the coefficient of $\tau ^{j}$ in both polynomial on the left of
(\ref{xxy}) and the polynomial on the right. So these polynomials are equal.
\end{proof}

Using the technique in the proof of Lemma\thinspace \ref{lem.fb}, one can also prove
the following lemma, which might be useful in calculations similar to the ones in our
paper (see also \cite{Pete}):

\begin{lemma}
The number of permutations $\mathfrak{s}\in \cS_{n}$ with $\mathfrak{%
s}(x)<x$ for exactly $j$ points $x\in \{1,2,\dotsc ,n\}$ is the Eulerian number
$\ebinom{n}{j}$ and the corresponding generating function is the Euler polynomial.
\end{lemma}

\section{Moments of quantum L\'{e}vy areas}


To evaluate the moments $\mathbb{E}\left[ \mathcal{\hat{B}}_{[a,b)}^{\left( \sigma \right)
}\right] ^{n}$\!, we need to calculate the number $w_{n}^{(\sigma)}$, as explained in 
(\ref{moment}). By \eqref{eq.111}, we have
\begin{equation}\label{eq.start}
w_{n}^{(\sigma)}dT\otimes dT\cdots \otimes \overset{(n)}{dT}
\,=\,\left(\left(\sum_{h\neq k}(-1)^{\des(h,k)}R_{h,k}\right)^{\!\!\!n\,}\right)_{\!\!\!(1,1,...,\overset{\left( n\right) }{1})}
\end{equation} 
with
\begin{equation}
R_{h,k}\,:=\,1\otimes \dots \otimes 1\otimes \overset{(h)}{\{d\hat{A}%
^{\left( \sigma \right) }\}}\otimes 1\otimes \dots \otimes 1\otimes \overset{(k)}{\{d\hat{A}%
^{\dagger \left( \sigma \right) }\}}\otimes 1\otimes \dots \otimes 1\,.  \label{40}
\end{equation}
As in the previous section,
\begin{equation}
\des(h,k)\,:=\,%
\begin{cases}
0 & \text{if $h\leq k$}, \\
1 & \text{if $h>k$},%
\end{cases}
\label{41}
\end{equation}%
and the $n$th power is based on the sticky shuffle product in $\mathcal{T}\left(
\mathcal{L}\right)$ and its extension to the n\textit{th} tensor power 
$\bigotimes\!^{(N)}\mathcal{T}(\mathcal{L})$, as described in \eqref{eq.power} for
$n=2$.

If we set $e:=(h,k)$, then we may also write $R_{e}$ for $R_{h,k}$ and $\des(e)$
for $\des(h,k)$. Using distributivity, 
this yields
\begin{equation}
w_{n}^{(\sigma)}dT\otimes dT\cdots \otimes \overset{(n)}{dT}
\,=\,\left(\sum\prod_{\ell=1}^n(-1)^{\des(e_\ell)}R_{e_\ell}\right)_{\!\!\!(1,1,...,\overset{\left(
n\right) }{1})} \,,\label{eq.sum}
\end{equation}
where the sum runs over all $n$-tuples $(e_1,e_2,\dotsc, e_n)$ of pairs $(h,k)$ with
$h\neq k$. We may imagine each pair $e_\ell=(h_\ell,k_\ell)$ as a directed edge, an
\emph{arc}, from $h_{\ell }$ to $k_{\ell }$. Each $n$-tuples $(e_1,e_2,\dotsc, e_n)$ is
then a directed labeled multigraph, or \emph{digraph}, on the vertex set
$V:=\{1,2,\dotsc,n\}$. It is important to keep track of the indices $\ell$ as \emph{labels}
of the arcs $e_{\ell }$, because our product is not commutative,
\begin{equation}
d\hat{A}^{\left( \sigma \right) }d\hat{A}^{\dagger \left( \sigma \right) }
\,=\,\sigma_{\!+}\,dT
\quad\text{and}\quad
d\hat{A}^{\dagger \left( \sigma \right)}d\hat{A}^{\left( \sigma \right) }
\,=\,\sigma_{\!-}\,dT\,,
\end{equation}
and
\begin{equation}
\bigl(\{d\hat{A}^{\left( \sigma \right) }\}\{d\hat{A}^{\dagger \left( \sigma \right) }\}\bigr)_{\!(1)}
\,=\,\sigma_{\!+}\,dT
\quad\text{and}\quad
\bigl(\{d\hat{A}^{\dagger \left( \sigma \right)}\}\{d\hat{A}^{\left( \sigma \right) }\}\bigr)_{\!(1)}
\,=\,\sigma_{\!-}\,dT\,.
\end{equation}
For example, in the case $n=4$, the two arcs $e_{1}=(1,2)$ and $e_{2}=(2,3)$
contribute
\begin{equation}\label{42}
\begin{split}
\bigl(&R_{e_1}R_{e_2}\bigr)_{\!(1,1,1,1)}\\
&\,=\,\bigl((\{d\hat{A}^{\left( \sigma \right)}\}\otimes \{d\hat{A} ^{\dagger \left( \sigma \right) }\}\otimes 1\otimes 1)(1\otimes
\{d\hat{A} ^{\left( \sigma \right) }\}\otimes \{d\hat{A}^{\dagger \left( \sigma \right)}\}\otimes 1)\bigr)_{\!(1,1,1,1)} \\
&\,=\,\bigl(\{d\hat{A}^{\left( \sigma \right)}\}1\bigr)_{\!(1)}\otimes
\bigl(\{d\hat{A} ^{\dagger \left( \sigma \right) }\}\{d\hat{A} ^{\left( \sigma \right) }\}\bigr)_{\!(1)}\otimes
\bigl(1\{d\hat{A}^{\dagger \left( \sigma \right)}\}\bigr)_{\!(1)}\otimes
\bigl(1\cdot1\bigr)_{\!(1)}\\
&\,=\,\sigma_{\!-}\,d\hat{A}^{\left( \sigma \right) }\otimes dT\otimes d\hat{A}^{\dagger \left( \sigma \right) }\otimes 1\,,
\end{split}
\end{equation}
while if the labels $1$ and $2$ are exchanged, $e_{1}=(2,3)$ and $e_{2}=(1,2)$, we
get
\begin{equation}\label{43}
\begin{split}
\bigl(&R_{e_1}R_{e_2}\bigr)_{\!(1,1,1,1)}\\
&\,=\,\bigl((1\otimes \{d\hat{A}^{\left( \sigma \right) }\}\otimes \{d\hat{A}^{\dagger \left( \sigma \right) }\}\otimes 1)(\{d\hat{A}^{\left( \sigma
\right) }\}\otimes \{d\hat{A}^{\dagger \left( \sigma \right) }\}\otimes 1\otimes1)\bigr)_{\!(1,1,1,1)} \\
&\,=\,\sigma _{\!+}\,d\hat{A}^{\left( \sigma \right) }\otimes dT\otimes d\hat{A}^{\dagger \left( \sigma \right) }\otimes 1\,.
\end{split}
\end{equation}

In order to calculate the coefficient $w_{n}^{\left( \sigma \right) }$ of $dT\otimes
dT\otimes \dotsm \otimes dT$ in \eqref{eq.start}, we need to retain only those
summands in \eqref{eq.sum} that contribute scalar multiple of $dT\otimes dT\otimes
\dotsm \otimes dT$. We may discard other summands. Hence, we do not have to sum
over all digraphs $(e_1,e_2,\dotsc, e_n)$. To see which ones we have to retain, let us
assume that $(e_1,e_2,\dotsc, e_n)$ yields a multiple of $dT\otimes dT\otimes \dotsm
\otimes dT$ in \eqref{eq.sum}. Since the $n$ copies of $d\hat{A}^{\left( \sigma \right) }$
and $n$ copies of $d\hat{A}^{\dagger \left( \sigma \right) }$ in the unexpanded product
$\prod_{\ell=1}^nR_{e_\ell}$ must yield $n$ copies of $dT$\!, one in each possible
position, each vertex of the digraph $(e_1,e_2,\dotsc, e_n)$ must have exactly one
incoming arc and one outgoing arc. Thus, $(e_1,e_2,\dotsc, e_n)$ must consist of
disjoint cyclically oriented cycles that cover $V$\!.
This allows us to view each arc $e_\ell=(h_\ell,k_\ell)$ as the assignment of a function
value, $h_\ell\mapsto k_\ell=:\mathfrak{s}(h_\ell)$. We obtain a fixed-point-free
permutation $\mathfrak{s}$ on $V=\{1,2,\dotsc ,n\}$. We obtain a second permutation
$\mathfrak{l}$ on $V$ by assigning to each label $\ell\in V$ the vertex $h_\ell$ from
which the arc $e_\ell=(h_\ell,k_\ell)$ originates, $\ell\mto h_\ell=:\cl(\ell)$. The pair
$(\cl,\cs)$ of permutations, $\mathfrak{l}$ in $\cS_{n}$ and $\mathfrak{s}$ in the
set $\cS_{n}^{\neq }$ of fixed-point-free permutations of $%
V=\{1,2,\dotsc ,n\}$, contains the full information about $(e_1,e_2,\dotsc, e_n)$.
Our construction describes a bijection $(e_1,e_2,\dotsc,
e_n)\lmto(\mathfrak{l},\mathfrak{s})$ from the set of digraphs $(e_1,e_2,\dotsc, e_n)$
that contribute a multiple of $dT\otimes dT\otimes\dotsm\otimes dT$ onto the set
$\cS_{n}\times\cS_{n}^{\neq}$. The edges $e_\ell$ of the digraph $(e_1,e_2,\dotsc,
e_n)$ can be recovered from $\cs$ and $\cl$ through the formula
\begin{equation}
e_\ell\,=\,\bigl(\cl(\ell),\cs(\cl(\ell))\bigr)\,,
\end{equation}
which describes the inverse bijection $(\mathfrak{l},\mathfrak{s})\lmto(e_1,e_2,\dotsc,
e_n)$. With this, the term $w_{n}^{(\sigma)}dT\otimes dT\otimes\dotsm\otimes dT$ in
\eqref{eq.sum} can be calculated as 
\begin{equation}
\begin{split}
w_{n}^{(\sigma)}\,dT\otimes&dT\otimes\dotsm\otimes dT\\
&=\,\left(\sum_{(\cl,\cs)\in\cS_{n}\times\cS_{n}^{\neq}}\,\prod_{\ell=1}^n(-1)^{\des(\cl(\ell),\cs(\cl(\ell))}R_{\cl(\ell),\cs(\cl(\ell))}\right)_{\!\!\!(1,1,...,\overset{\left(
n\right) }{1})}\\
&=\,\sum_{(\cl,\cs)\in\cS_{n}\times\cS_{n}^{\neq}}\!\!\!\!\sn(\mathfrak{s})\left(\,\prod_{\ell=1}^nR_{\cl(\ell),\cs(\cl(\ell))}\right)_{\!\!\!(1,1,...,\overset{\left(
n\right) }{1})}\,,
\end{split} 
\end{equation}%
where, for every $\cl\in\cS_n$,
\begin{equation}
\sn(\mathfrak{s})\,:=\,\prod_{j=1}^{n}(-1)^{\des(j,\mathfrak{s}(j))}
\,=\,\prod_{\ell=1}^{n}(-1)^{\des(\cl(\ell),\cs(\cl(\ell))}\,.\label{eq.snl}
\end{equation}%

We have to consider the product $\prod_{\ell=1}^nR_{\cl(\ell),\cs(\cl(\ell))}$, for every
fixed $(\cl,\cs)\in\cS_{n}\times\cS_{n}^{\neq}$. In this product, one $dT$ is produced in 
each position $j$, and it comes either with the scalar factor $\sigma_{\!+}$ or with
$\sigma_{\!-}$. If $dT$ arrives as $d\hat{A}^{(\sigma)}d\hat{A}^{\dagger(\sigma)}$ then
we get $\sigma_{\!+}$ as scalar factor, if it arrives as
$d\hat{A}^{\dagger(\sigma)}d\hat{A}^{(\sigma)}$ then we get $\sigma_{\!-}$. To
examine how the $dT$ in position $j$ arrives, let $\ell_1:=\cl^{-1}(j)$ and
$\ell_2=\cl^{-1}(\cs^{-1}(j))$, then $R_{\cl(\ell_1),\cs(\cl(\ell_1))}$ contributes a
$d\hat{A}^{(\sigma)}$ in position $j$ as $\ell_1$th factor, and
$R_{\cl(\ell_2),\cs(\cl(\ell_2))}$ contributes a $d\hat{A}^{\dagger(\sigma)}$ in position
$j$ as $\ell_2$th factor. So, if $\cl^{-1}(\cs^{-1}(j)))>\cl^{-1}(j)$ then $\ell_2>\ell_1$ and
the $d\hat{A}^{\dagger(\sigma)}$ comes after the $d\hat{A}^{(\sigma)}$\!, yielding a
$\sigma_{\!+}$ as scalar factor. In general, the scalar factor of the $dT$ in position $j$
is $\sigma_{\!-}\,(\sigma_{\!+}/\sigma_{\!-})^{\des(\cl^{-1}(\cs^{-1}(j)),\cl^{-1}(j))}$.
Therefore,
\begin{equation}
w_{n}^{(\sigma)}\,=\,\sigma_{\!-}^n\!\!\!\!\sum_{(\cl,\cs)\in\cS_{n}\times\cS_{n}^{\neq}}\!\!\!\!\sn(\cs)
\prod_{j=1}^{n}(\sigma_{\!+}/\sigma_{\!-})^{\des(\cl^{-1}(\cs^{-1}(j)),\cl^{-1}(j))}\,.
\end{equation}
We substitute $\cs(j)$ for $j$ and $\cl^{-1}$ for $\cl$, and obtain
the following theorem:

\begin{theorem}\label{sz.wn}
\begin{equation*}
w_{n}^{(\sigma) }\,=\,\sigma_{\!-}^{n}\sum_{\mathfrak{s}\in \cS_{n}^{\neq }}\sn(\mathfrak{s})
\sum_{\mathfrak{l}\in \cS_{n}}\prod_{j=1}^{n}\tau^{\des(\mathfrak{l}(j),\mathfrak{l}(\mathfrak{s}(j)))}\,,
\end{equation*}%
where $\tau :=\sigma _{\!+}/\sigma _{\!-}$.
\end{theorem}


In this expression for $w_{n}^{\left( \sigma \right) }$ there are many terms that cancel
against each other when we carry out the sum. To remove these unnecessary
summands and bundle together equal terms, we first study the inner sum
\begin{equation}
w_{n}^{\,\cs}(\tau )\,:=\,\sum_{\cl\in\cS_{n}}\prod_{j=1}^{n}\tau^{\des(\cl(j),\cl(\cs(j)))}\,,  \label{46}
\end{equation}%
for a fixed $\cs\in\cS_{n}^{\neq }$. Initially, for simplicity, also assume that there
is only one cycle, of length $n$ in $\mathfrak{s}$. In cycle notation, $\mathfrak{s}=(\mathfrak{s}_{1},\mathfrak{s}%
_{2},\dotsc ,\mathfrak{s}_{n})$ with $\mathfrak{s}_{2}=\mathfrak{s}(%
\mathfrak{s}_{1})$, $\mathfrak{s}_{3}=\mathfrak{s}(\mathfrak{s}_{2})$, etc. In this
particular case, by Lemma\,\ref{lem.cdes},
\begin{equation}
\begin{split}
w_{n}^{\,\mathfrak{s}}(\tau )\,&=\,\sum_{\mathfrak{l}\in \cS_{n}}
\prod_{\ell=1}^{n}\tau^{\des(\mathfrak{l}(\cs_\ell),\mathfrak{l}(\cs_{\ell+1}))}\\
&=\,\sum_{\mathfrak{l}\in \cS_{n}}
\tau^{\cdes(\mathfrak{l}(\mathfrak{s}_{1}),\mathfrak{l}(\mathfrak{s}_{2}),\dotsc ,\mathfrak{l}(\mathfrak{s}_{n}))}\\
\,&=\,\sum_{\mathfrak{r}\in \cS_{n}}\tau ^{\cdes(\mathfrak{r}(1),\mathfrak{r}(2),\dotsc ,\mathfrak{r}(n))}\\
\,&=\,n\,\tau S_{n-1}(\tau )\,,\label{eq.zyc}
\end{split}
\end{equation}%
where $\cS_n(\tau)$ is the Euler polynomial and
$\cdes(\mathfrak{r}(1),\mathfrak{r}(2),\dotsc ,\mathfrak{r}(n))$ denotes the number of
descends of the sequence
$(\mathfrak{r}(1),\mathfrak{r}(2),\dotsc,\mathfrak{r}(n),\mathfrak{r}(1))$.

Formula\,\eqref{eq.zyc} holds only for cyclic permutations $\mathfrak{s}\in
\cS_{n}^{\neq }$. For the general case, suppose that $\mathfrak{s}$ has
$k=k(\mathfrak{s})$ cycles of lengths $n_{1},n_{2},\dotsc ,n_{k}$, say where $2\leq
n_{1}\leq n_{2}\leq \dotsb \leq n_{k}$ and $n_{1}+n_{2}+\dotsb +n_{k}=n$. We say that
$(n_1,n_2,\dotsc,n_k)$ is the \emph{typ} of $\cs$ and write
$\mathfrak{s}\in\cC_{n_{1},n_{2},\dotsc ,n_{k}}$. We have to split the product in the
definition (\ref{46}) of $w_{n}^{\,\mathfrak{s}}(\tau )$ into $k$ parts correspondingly. If
$C_\ell$ denotes the set of the $n_\ell$ elements of the $\ell$th cycle of the fixed given
$\cs\in\cC_{n_{1},n_{2},\dotsc ,n_{k}}$\!, then
\begin{equation}
\prod_{j=1}^{n}\tau^{\des(\cl(j),\cl(\cs(j)))}\,=\,\prod_{\ell=1}^{k}\,\prod_{j\in C_\ell}\tau^{\des(\cl_\ell(j),\cl_\ell(\cs(j)))}\,,
\end{equation}
where $\cl_\ell$ is the restriction of $\cl$ to $C_\ell$, so that
$\cl=\cl_1\cup\cl_2\cup\dotsb\cup\cl_n$. The range of each $\cl_\ell$ can be any
subset $N_\ell\sb V$ of $n_\ell$ elements, provided only that all the subsets $N_\ell$
together form an ordered partition $(N_1,N_2,\dotsc,N_k)$ of $V$. We want to describe
the set $\cS_n$ of permutations $\cl$ in terms of smaller bijections
$\cl_\ell:C_\ell\rightarrow N_\ell$. Let $\mathcal{N}$ denotes the set of all partitions
$N:=(N_1,N_2,\dotsc,N_k)$ of $V$ into $k$ blocks
$N_\ell$ with $|N_\ell|=n_\ell$, let $B_\ell(N)$ 
be the set of bijections from $C_\ell$ to $N_\ell$, and let $B(N):=B_1(N)\times
B_2(N)\times\dotsm\times B_n(N)$. With this, the set of permutations $\cS_n$ is
partitioned as
\begin{equation}
\cS_n\,=\,\bigcup_{N\in\mathcal{N}}\{\cl_1\cup\cl_2\cup\dotsb\cup\cl_n\,|\,(\cl_1,\cl_2,\dotsc,\cl_n)\in B(N)
\}\,.
\end{equation}
From that disjoint union we get
\begin{equation}
\begin{split}
w_{n}^{\,\cs}(\tau )\,&=\,\sum_{N\in\mathcal{N}}\,\sum_{\cl\in B(N)}\,\prod_{\ell=1}^{k}\,
\prod_{j\in C_\ell}\tau^{\des(\cl_\ell(j),\cl_\ell(\cs(j)))}\\
&=\,\sum_{N\in\mathcal{N}}\,\prod_{\ell=1}^{k}\,
\sum_{\cl_\ell\in B_\ell(N)}\,\prod_{j\in C_\ell}\tau^{\des(\cl_\ell(j),\cl_\ell(\cs(j)))}\,.
\end{split}
\end{equation}
Here, for all $N\in\mathcal{N}$, the inner sum is
\begin{equation}
\sum_{\cl_\ell\in B_\ell(N)}\,\prod_{j\in C_\ell}\tau^{\des(\cl_\ell(j),\cl_\ell(\cs(j)))}\,=\,n_\ell\,\tau S_{n_\ell-1}(\tau )\,,
\end{equation}
by \eqref{eq.zyc}, because the names of the elements in $C_\ell$ and $N_\ell$ do not
matter.
Every fixed set $N_\ell$ of $n_\ell$ different numbers is linearly ordered and produces
the same statistic for the cyclic descents, if we consider all sequences that can be
arranged using all elements of $N_\ell$.
Using
\begin{equation}
|\mathcal{N}|\,=\,\binom{n}{n_{1},n_{2},\dotsc ,n_{k}}\,:=\,\frac{n!}{n_{1}!n_{2}!\dotsm n_{k}!}\,,
\end{equation}
we obtain
\begin{equation}
w_{n}^{\,\mathfrak{s}}(\tau )\,:=\,\binom{n}{n_{1},n_{2},\dotsc ,n_{k}}%
\prod_{j=1}^{k}n_{j}\,\tau S_{n_{j}-1}(\tau )\,,  \label{eq.wns}
\end{equation}
where $(n_1,n_2,\dotsc,n_k)$ is still the typ of $\cs$, i.e.\
$\mathfrak{s}\in\cC_{n_{1},n_{2},\dotsc ,n_{k}}$.

We can now calculate $w_n^{(\sigma)}$ out of \eqref{eq.wns} and
Theorem\,\ref{sz.wn}. Since we have the disjoint union
\begin{equation}
\cS_n^{\neq}\,=\bigcup_{\!\!\atop{n_{1}+n_{2}+\dotsb+n_{k}=n}{2\leq n_1\leq n_2\leq\dotsb\leq n_k}}\!\!\!\!\!\!\cC_{n_{1},n_{2},\dotsc ,n_{k}}\,,
\end{equation}
we get
\begin{equation}
\begin{split}
w_{n}^{(\sigma) }
&=\,\sigma_{\!-}^{n}\sum_{\mathfrak{s}\in \cS_{n}^{\neq }}\sn(\mathfrak{s})
\binom{n}{n_{1},n_{2},\dotsc ,n_{k}}\prod_{\ell=1}^{k}n_{\ell}\,\tau S_{n_{\ell}-1}(\tau )\\
&=\,\sigma_{\!-}^{n}\!\!\!\!\!\!\!\!\sum_{\!\!\atop{n_{1}+n_{2}+\dotsb+n_{k}=n}{2\leq n_1\leq n_2\leq\dotsb\leq n_k}}\,
\sum_{\mathfrak{s}\in\cC_{n_{1},n_{2},\dotsc ,n_{k}}\!\!\!\!\!\!\!\!\!\!\!\!}\!\!\!\!\sn(\mathfrak{s})
\binom{n}{n_{1},n_{2},\dotsc ,n_{k}}\prod_{\ell=1}^{k}n_{\ell}\,\tau S_{n_{\ell}-1}(\tau )\,.\\
&=\,\sigma_{\!-}^{n}\!\!\!\!\!\!\!\!\sum_{\!\!\atop{n_{1}+n_{2}+\dotsb+n_{k}=n}{2\leq n_1\leq n_2\leq\dotsb\leq n_k}}
\binom{n}{n_{1},n_{2},\dotsc ,n_{k}}\biggl(\prod_{\ell=1}^{k}n_{\ell}\,\tau S_{n_{\ell}-1}(\tau )\biggr)\!\!
\sum_{\mathfrak{s}\in\cC_{n_{1},n_{2},\dotsc ,n_{k}}\!\!\!\!\!\!\!\!\!\!\!\!}\!\!\!\!\sn(\mathfrak{s})\,.
\end{split}
\end{equation}%
Now, Theorem\,\ref{sz.snd} shows that $w_{n}^{(\sigma) }=0$ for odd $n$, and that for
even $n$, $n=2m$,
\begin{equation}
\begin{split}
w_{2m}^{(\sigma) }
\,&=\,(-1)^{m}\sigma_{\!-}^{2m}\!\!\!\!\!\!\!\!\sum_{\!\!\atop{m_{1}+m_{2}+\dotsb+m_{k}=m}{1\leq m_1\leq m_2\leq\dotsb\leq m_k}}\\
&\bbinom{2m}{2m_{1},2m_{2},\dotsc ,2m_{k}}\binom{2m}{%
2m_{1},2m_{2},\dotsc ,2m_{k}}\prod_{\ell=1}^{k}2m_{\ell}A_{2m_{\ell}-1}\tau
S_{2m_{\ell}-1}(\tau )\,.
\end{split}
\end{equation}
Using Lemma\,\ref{lem.prt}, we obtain the following theorem:

\begin{theorem}
For odd $n$, $w_n^{(\sigma)}=0$. For even $n$, $n=2m>0$, we have
\begin{equation*}
w_{2m}^{\left( \sigma \right) }\,=\,(-1)^{m}(2m)!^{2}\sum \frac{\mathfrak{%
\sigma }_{\!-}^{2m-k}\mathfrak{\sigma }_{\!+}^{k}}{k_{1}!k_{2}!\dotsm
k_{r}!}\prod_{j=1}^{k}\frac{A_{2m_{j}-1}}{2m_{j}(2m_{j}{-}1)!^{2}}%
\,S_{2m_{j}-1}(\mathfrak{\sigma }_{\!+}/\mathfrak{\sigma }_{\!-})\,,
\end{equation*}%
where the sum runs over all partitions $m_{1}+m_{2}+\dotsb +m_{k}=m$ with $1\leq
m_{1}\leq m_{2}\leq \dotsb \leq m_{k}$, and where $k_{1}$, $k_{2}$, \dots , $%
k_{r}$ are the corresponding multiplicities of the different elements in the multi-set
$\{m_{1},m_{2},\dotsc ,m_{k}\}.$
\end{theorem}

From this and Equation\,(\ref{moment}), we derive our final result:

\begin{theorem}
The nonzero moments of the quantum L\'{e}vy area $\mathcal{\hat{B}}%
_{[a.b)}^{\left( \sigma \right) }$ are%
\begin{equation*}
\begin{split}
\mathbb{E}&\left[ \left( \mathcal{\hat{B}}_{[a.b)}^{\left( \sigma \right)
}\right) ^{2m}\right]\\
&=(2m)!\left( b-a\right) ^{2m}\sum \frac{\mathfrak{\sigma }_{\!-}^{2m-k}%
\mathfrak{\sigma }_{\!+}^{k}}{k_{1}!k_{2}!\dotsm k_{r}!}\prod_{j=1}^{k}%
\frac{A_{2m_{j}-1}}{2m_{j}(2m_{j}{-}1)!^{2}}\,S_{2m_{j}-1}(\mathfrak{\sigma }%
_{\!+}/\mathfrak{\sigma }_{\!-})\,,  \label{final}
\end{split}
\end{equation*}
where the sum runs over all partitions $m_{1}+m_{2}+\dotsb +m_{k}=m$ with $1\leq
m_{1}\leq m_{2}\leq \dotsb \leq m_{k}$, and where $k_{1}$, $k_{2}$, \dots , $%
k_{r}$ are the corresponding multiplicities of the different elements in the multi-set
$\{m_{1},m_{2},\dotsc ,m_{k}\}.$ The $A_{n}$ are Euler zigzag numbers and the
$S_{n}$ are Euler polynomials.
\end{theorem}

\section{\protect\bigskip The classical limit}

\bigskip We calculate the limit of $\mathbb{E}\left[ \left( \mathcal{\hat{B}}%
_{[a.b)}^{\left( \sigma \right) }\right)^{\!2m}\right] $ as $\sigma \rightarrow \infty ,$ or
equivalently, $\sigma _{\!+}\rightarrow \frac{1}{2},$ $\sigma
_{\!-}\rightarrow \frac{1}{2}.$ Putting $\sigma _{\!+}=$ $\sigma _{\!-}=\frac{1}{2}$ in  (\ref%
{final}), we get%
\begin{eqnarray}
&&\!\!\!\!\!\!\!\!\!\!\!\!\!\!\!\underset{\sigma \rightarrow \infty }{\lim }\mathbb{E}\left[ \left(
\mathcal{\hat{B}}_{[a.b)}^{\left( \sigma \right) }\right)^{\!2m}\right]  \\\notag
&=\,&(2m)!\left( \frac{b-a}{2}\right)^{\!2m}\sum \frac{1}{k_{1}!k_{2}!\dotsm
k_{r}!}\prod_{j=1}^{k}\frac{A_{2m_{j}-1}}{2m_{j}(2m_{j}{-}1)!^{2}}%
\,S_{2m_{j}-1}(1)\,,
\end{eqnarray}%
where the sum runs over all partitions $m_{1}+m_{2}+\dotsb +m_{k}=m$, and where
$k_{1}$, $k_{2}$, \dots , $k_{r}$ are the corresponding multiplicities of the elements of
the multi-set $\{m_{1},m_{2},\dotsc ,m_{k}\}$. But, by the definition of the Euler
polynomial,
\begin{equation}
S_{n}(1)\,:=\,\sum_{\cs\in \cS_{n}}1^{\des(\cs(1),\cs(2),\dotsc ,\cs(n))}\,=\,|\cS_{n}|\,=\,n!\,.  \label{61}
\end{equation}%
So, using Lemma\,\ref{lem.fbd} and Lemma\,\ref{lem.prt}, we see that%
\begin{eqnarray}
\underset{\sigma \rightarrow \infty }{\lim }\mathbb{E}\left[ \left(
\mathcal{\hat{B}}_{[a.b)}^{\left( \sigma \right) }\right)^{\!2m}\right]
&\,=\,&(2m)!\left( \frac{b-a}{2}\right)^{\!2m}\sum \frac{1}{k_{1}!k_{2}!\dotsm
k_{r}!}\prod_{j=1}^{k}\frac{A_{2m_{j}-1}}{(2m_{j})!}\notag \\
&\,=\,&\left(\frac{ b-a}{2}\right)^{\!2m}\!A_{2m}\,.
\end{eqnarray}%
This result is in agreement with the main theorem in \cite{LeWi}.

\end{document}